\documentclass{birkjour}
\newtheorem{theorem}{Theorem}
\newtheorem{lemma}[theorem]{Lemma}

\theoremstyle{definition}
\newtheorem{definition}[theorem]{Definition}
\newtheorem{remark}[theorem]{Remark}

 \newtheorem{proposition}[theorem]{Proposition}%
 \theoremstyle{definition}
 
 \theoremstyle{remark}

 \numberwithin{equation}{section}
\def\R{{\mathbb{R}}}
\def\N{{\mathbb{N}}}

\newcommand{\ds}{\displaystyle}
\renewcommand{\theequation}{\arabic{section}.\arabic{equation}}

\begin{document}

%
%
%
%
%
%
%
%
%

\title[Non--linear Degenerate Elliptic Equations]
 {The influence of the Hardy potential and a Convection Term on a  Non--linear Degenerate Elliptic Equations}

\author[Fessel Achhoud]{Fessel Achhoud}

\address{%
Faculté des Sciences et Techniques MISI Laboratory, Hassan
First University of Settat, B.P. 577 Settat, 26000, Morocco.}

\email{f.achhoud@uhp.ac.ma}

\author{Abdelkader Bouajaja}
\address{Hassan First University,
Faculty of Economics and Management,
B.P. 577, Settat, Morocco.}
\email{abdelkader.bouajaja@uhp.ac.ma}
\author{Hicham Redwane}
\address{Hassan First University,
Faculty of Economics and Management,
B.P. 577, Settat, Morocco.}
\email{hicham.redwane@uhp.ac.ma}



\begin{abstract}
This paper is devoted to prove existence of renormalized solutions for a class of non--linear degenerate elliptic equations involving a non--linear convection term, which satisfies a growth properties, and a Hardy
potential. Additionally, we assume that the right-hand side  is an $L^m$ function, with $m\geq 1$.
\end{abstract}

\maketitle
\section{Introduction}\label{s1}
 Let $\mathcal{W}$ denote a bounded  open subset of $\R^N$($N\geq 3$) such that $0\in \mathcal{W}$.
 Consider the following model of  nonlinear elliptic problem with principal part having degenerate coercivity
\begin{equation}\label{prototype}
\begin{cases}-\ds\operatorname{div}\left(\frac{\vert\nabla v\vert^{p-2}\nabla v}{(1+\vert v\vert)^{\theta(p-1)}}+c_0(x)\vert v\vert^{\lambda-1}v\right)=\gamma \frac{\vert v\vert^{s-1}v}{\vert x\vert^{p}}+f & \text {in } \mathcal{W},  \\ u=0 & \text {on } \partial \mathcal{W},\end{cases}
\end{equation}
here $p\in (1,+\infty)$ and $\theta,\,\lambda,\,\gamma$ and $s$ are positive numbers, $c_0(x)\in L^{\frac{N}{p-1}}(\mathcal{W})$ and $f$ is in $L^m(\mathcal{W})$ with $m\geq 1.$ 
Let  us  assume  that  the operator  has  no  convection term and no Hardy potential,  i.e.  $\theta=c_0=\gamma=0$, in this 
case  the difficulties  in  studying  problem  \eqref{prototype} are  due  only to  the  right-hand  side $f$.
 We recall that in the classical case $\theta=0$, such kind of problems with convection  term  were studied well in the literature in a different frameworks for an exhaustive review
of this topic, we refer to \cite{17,D, 16, MP, M, 35, T, Z}. Moreover, we recall also the works \cite{20, 21, 32, 40, 36}  where the classic boundary value problems involving the Hardy potential.

Given $k > 0$ and $\forall n\in \N^*$, denote by $\mathcal{T}_{k}$ the truncation function at level $\pm k$ define as 
\begin{equation*}
\mathcal{T}_{k}(t)=\min \{k, \max \{-k, t\}\}, \quad \forall t\in\R.
\end{equation*}
 It is well known that  the framework of renormalized or entropy solution makes a sense according to the following definition of the weak gradient
\begin{definition}(See  \cite{38},  Lemma 2.1). \label{dff1}
If $v$ is a measurable function defined on $\mathcal{W}$ that is almost everywhere finite and satisfies $\mathcal{T}_{k}(v)\in W_0^{1,p}(\mathcal{W})$ for all $k>0$, then there exists a unique measurable function $w:\mathcal{W}\rightarrow \mathbb{R}^N$ such that 
\begin{equation*}
\nabla \mathcal{T}_{k}(v)=w\chi_{\{\vert u\vert\leq k\}}.
\end{equation*}
 Thus, we can define the generalized gradient $\nabla v$ of $v$ as this function $w$, and denote $\nabla v=w$.
\end{definition}
The same reasoning applies if we deal with the degenerate case, on condition that $u$ is finite almost everywhere in $\mathcal{W}$ and such that
\begin{equation}\label{grd}
\nabla \mathcal{T}_{k}(v)\in \left(L^{p}(\mathcal{W})\right)^N\quad \text{for every $k>0$}.
\end{equation}

The degenerate case was firstly studied in \cite{B}. In this paper,  Boccardo and al have studied the existence and regularity for the following quasi-linear elliptic problem 
\begin{equation}\label{crv1}
\begin{aligned}
\begin{cases}-\ds\operatorname{div}(A(x, v) \nabla v)=f & \text { in } \mathcal{W}, \\ u=0 & \text { on } \partial \mathcal{W},\end{cases}
\end{aligned}
\end{equation}
here $f$ is assumed to be in $L^{m}(\mathcal{W})$ with $m \geq 1$,
 and $A(x, t): \mathcal{W} \times \mathbb{R} \rightarrow$ $\mathbb{R}$ is a  measurable function with respect to $x$ for every $t \in \mathbb{R}$, and continuous function with respect to $t$ for almost every $x \in \mathcal{W}$, satisfying the following condition: 
there exist $\theta \in [0,1],$  $\alpha,\,\beta\in (0,\infty)$  such that
\begin{equation*}
\frac{\alpha}{(1+\vert t\vert)^\theta} \leq A(x, t) \leq \beta,\quad \text{for a.e. $x\in \mathcal{W}$,  $\forall t\in \R$.}
\end{equation*}
Moreover, in the paper \cite{BG} the authors demonstrated the existence of a renormalized solutions for the problem \eqref{crv1} with datum  $f \in L^{1}(\mathcal{W})$ and $A(x, t): \mathcal{W} \times \mathbb{R} \rightarrow \mathbb{R}^{N \times N}$ is a Carathéodory function with values in the space of matrices on $\mathbb{R}$ and is not assumed to be symmetric. 
A result on the existence and regularity of weak and entropy solutions is obtained, by Alvino and al  in \cite{104}, for a nonlinear degenerate elliptic problem of the form
\begin{align*}
\begin{cases}-\ds\operatorname{div}\left(\frac{\vert\nabla v\vert^{p-2} \nabla v}{(1+\vert v\vert)^{\theta(p-1)}}\right)=f & \text { in } \mathcal{W}, \\ u=0 & \text { on } \partial \mathcal{W},\end{cases}
\end{align*}
where  $f$ is a measurable function in $L^{m}(\mathcal{W})$ with $m \geq 1$.

One of the main points that we stress in this paper  is to analyze the interaction between  the convection term and  the one singular at the origin, the so-called Hardy potential, to obtain the existence  of a renormalized solution for the problem \eqref{prototype}. The influence of Hardy potential in elliptic problems has been studied in several papers (see for example the book \cite{36} for a more general framework).
    
Actually, if $c_0(x)=0$ and $f$ is a nonnegative function in $L^m(\mathcal{W})$ with $m\geq 1,$ the authors  established, in \cite{22}, an existence and non-existence result of non-negative renormalized solutions for a nonlinear degenerate elliptic problem of the form
\begin{equation*}
\begin{aligned}
\begin{cases}
\ds-\operatorname{div}\left(\frac{ \vert \nabla v\vert^{p-2}\nabla v}{(d(x)+\vert v\vert)^{\theta(p-1)}}\right)=\gamma \frac{\vert v\vert^{s}}{\vert x\vert^p}+f & \text{in}\, \mathcal{W},\\
u\geq 0 & \text{in}\, \mathcal{W},\\
u=0 &\text{on} \, \partial \mathcal{W},
\end{cases}
\end{aligned}
\end{equation*}
where $\gamma$ and $s$ are positive numbers and 
$d: \mathcal{W}\rightarrow \left(0,+\infty\right)$ is a bounded measurable function. 
In addition, one of the most interesting phenomena that exhibit this problem if $s=(1-\theta)(p-1)$ and $f\in L^1(\mathcal{W})$ is the non-existence of  solutions.  This non-existence result can be illustrated by considering the following simplest problem, studied in \cite{32},
\begin{equation*}
\begin{aligned}
\begin{cases}
\ds-\Delta v=\gamma \frac{v}{\vert x\vert^2}+f & \text{in}\, \mathcal{W},\\
u=0 &\text{on} \, \partial \mathcal{W}.
\end{cases}
\end{aligned}
\end{equation*}


On the other hand, in the case where  $f$ is a Radon measure with bounded variation defined on $\mathcal{W}$,  T. Del Vecchio and M.R. Posteraro demonstrated, using the symmetrization method, the existence of weak solutions for a class of nonlinear and  noncoercive problem involving a lower order term, whose prototype is
\begin{equation}\label{GG}
\begin{cases}
- \operatorname{div}(\vert\nabla v\vert^{p-2}\nabla v + c_0(x)\vert v\vert^{\lambda}) + d(x)\vert\nabla v\vert^\mu = f& \text{in } \mathcal{W}, \\
u = 0 & \text{on } \partial \mathcal{W},
\end{cases}
\end{equation}
in this context, the functions $d(x)$ and $c_0(x)$ belong to $L^N(\mathcal{W})$ and $L^{\frac{N}{p-1}}(\mathcal{W})$ respectively. Moreover, in the case when $\gamma=\mu=p-1$ they supposed that $\Vert d(x)\Vert_{L^N(\mathcal{W})}$ or $\Vert c_0(x)\Vert_{L^{\frac{N}{p-1}}(\mathcal{W})}$ is small enough.  The most delicate case was to obtain a priori estimate for $v$ and $\nabla v$ in the case when $\left\Vert  c_0(x)\right\Vert _{L^{\frac{N}{p-1}}(\mathcal{W})}$ is not small.
Recently, in the paper \cite{16}, O. Guibé and A. Mercaldo studied the problem  \eqref{GG} in the general framework of  Lorentz spaces. The authors successfully demonstrated the existence of renormalized solutions under the conditions  $0 \leq \mu, \lambda \leq p-1$ and $\Vert c_0(x)\Vert_{L^{\frac{N}{p-1}}(\mathcal{W})}$ is not small. This is done
by proving the following uniform estimate 
\begin{equation}\label{Nn}
\ds \forall \eta >0, \exists \nu_\eta>0 \;\;\; meas\lbrace\vert v \vert> \nu_\eta \rbrace\leq \frac{1}{\eta^p},
\end{equation} 
 which allowed them to derive the estimate \eqref{grd} which implies an estimates on $\ds\vert v\vert^{p-1}$ and $\ds\vert\nabla v\vert^{p-1}$, thanks to the lemma \ref{lapes}, in some Lorentz-Marcinkiewicz space.

 However, in the case $\theta \neq 0$ and $\gamma \neq 0$, one can not prove \eqref{Nn} but instead one only has  
 \begin{equation*}
\ds \forall \eta >0, \exists \bar{\nu}_\eta>0 \;\;\; meas\lbrace\vert \widetilde{\varrho}(v) \vert> \bar{\nu}_\eta \rbrace\leq \frac{1}{\eta^p}.
\end{equation*} 
with  $\widetilde{\varrho}(t)$ denotes the primitive of a decreasing continuous function given by 
\begin{equation*}
\varrho(t)=\frac{1}{(1+\vert t \vert)^{\theta}}\quad \quad \theta \in [0,1),
\end{equation*}
 which satisfies the following behavior at $\infty$
\begin{equation}\label{jj}
\lim_{\vert t\vert\rightarrow + \infty} \widetilde{\varrho}(t)=\pm \infty.
\end{equation}
 Moreover, there exists a positive constant $\widetilde{C}>0$ and a positive real number $k_0$  such that for every $\vert t\vert>k_0$, one has 
 \begin{equation}\label{jjj}
 \ds\frac{\vert  t\vert ^\lambda}{(1+\vert \widetilde{\varrho}(t)\vert )^{p-1}}\leq \widetilde{C}.
 \end{equation}
 
 We stress that the method used in \cite{16}, is not apply directly in our case since the right hand side of our problem involving the Hardy potential.  In order to overcome this difficulty, we prove an $L^1-$estimate on Hardy potential term  by arguing as in \cite{22}.
This is enough, thanks to \eqref{jj} and \eqref{jjj}, to ensure that $u$ satisfies \eqref{Nn}.

Finally we explicitly remark that if we consider $0 < s, \lambda < (1-\theta)(p-1)$. Thanks to \eqref{jjj}, the following inequalities hold
\begin{align*}
\ds \int_\mathcal{W} \frac{\left\vert v\right\vert ^{s}}{\vert x\vert ^p}\;dx\leq \widetilde{C} \int_\mathcal{W} \frac{(1+\vert \widetilde{\varrho}(v)\vert)^{p-1}}{\vert x\vert^p}\;dx,
\end{align*}
and
\begin{align*}
\ds\int_\mathcal{W} c_0(x)\vert v\vert^\lambda \;dx\leq  \widetilde{C} \int_\mathcal{W} c_0(x)(1+\widetilde{\varrho}(v))^{p-1}\;dx.
\end{align*}
Therefore, the equation \eqref{prototype} may be equivalently written as
\begin{equation}\label{problem2}
\begin{cases}-\ds\operatorname{div}\left(\vert\nabla \widetilde{\varrho}(v)\vert^{p-2}\nabla \widetilde{\varrho}(v)+\widetilde {c}_0(x)\vert\widetilde{\varrho}(v)\vert^{p-1}\right)=\gamma \ds\frac{\vert \widetilde{\varrho}(v)\vert^{p-1}}{\vert x\vert^{p}}+g & \text {in } \mathcal{W},  \\ \widetilde{\varrho}(v)=0 & \text {on } \partial \mathcal{W},\end{cases}
\end{equation}
such that  $\widetilde{c}_0(x) \in L^{\frac{N}{p-1}}(\mathcal{W})$ and  the 
%
 the right-hand side, $g\in L^m(\mathcal{W})$ with $m> 1$.
In general, the problem \eqref{problem2} is not coercive and has no weak solution when $g \in L^1(\mathcal{W}), \,\gamma >0$.
 Thus, in the present paper, we face the two difficulties arise from the presence of both the non--linear convection term and Hardy potential.
%
%
This means that we will be dealing with all the difficulties previously described, at the same time. To our knowledge, the exploration of the combined impact of the non--linear convection term and the Hardy potential has not been undertaken before.

For ease of reading, in the Sec. 2, we recall some well-know preliminaries, properties and definitions of the Lorentz-Marcinkiewicz space, as well as we set our main assumptions. While in Sec. 3 we give the proof of the existence  result.

\section{Some preliminaries and definitions}\label{s3}
Now, we give  some basic tools for functional analysis that we will use in our study.  
The Lorentz space $L^{q,r}(\mathcal{W})$ is the space of Lebesgue measurable functions such that for any  $(q,r)\in (1,\infty)^2$
\begin{equation*}
\Vert f \Vert_{L^{q, r}(\mathcal{W})}=
\ds\left(\int_{0}^{meas(\mathcal{W})}\left[f^{*}(t) t^{\frac{1}{q}}\right]^{r} \frac{d t}{t}\right)^{1 / r}<+\infty.
\end{equation*}
Where, $f^*$ stands for  the decreasing rearrangement of the function $f$ which defined by
\begin{equation*}
f^{*}(t)=\inf \{r \geq 0: \text { meas }\{x \in \mathcal{W}:\vert f(x)\vert>r\}<t\} \quad t \in[0,meas(\mathcal{W})].
\end{equation*}
We also mentioned that, for any $q\in[1,+\infty)$ the Lorentz-Marcinkiewicz space $L^{q, \infty}(\mathcal{W})$ is the set of measurable functions $f:\mathcal{W}\rightarrow \R$  such that
\begin{equation}\label{nrm}
\ds\Vert f\Vert_{L^{q, \infty}(\mathcal{W})}=\sup _{t} t\left(\text { meas }\{x \in \mathcal{W}:\vert f\vert>t\}\right)^{\frac{1}{q}}<+\infty.
\end{equation}
Moreover for any $z$ and $q$ such that $1 \leq q<r<z \leq+\infty$, the following chain of continuous inclusions in Lebesgue spaces holds true
\begin{equation}\label{ceb}
L^{z}(\mathcal{W}) \subset L^{r, \infty}(\mathcal{W}) \subset L^{q}(\mathcal{W}) \subset L^{1}(\mathcal{W}).
\end{equation}
For references about rearrangements see, for example, \cite{29}.

 The following lemmas introduce some well-known inequalities that will be very useful in a number of situations such as a priori estimates.
\begin{lemma}(Poincaré's inequality).
 Suppose $p \in [1,N)$ and $v \in W^{1, p}_0\left(\mathcal{W}\right) .$ Then there exist a constant $c(N,p)$ such that
 \begin{equation}
 \Vert v \Vert_{L^{p}(\mathcal{W})} \leq c(N,p)\Vert \nabla v \Vert_{L^{p}(\mathcal{W})},
 \end{equation}
\end{lemma}
\begin{lemma}(Sobolev's inequality).
 Suppose $p \in [1,N)$ and $v \in W^{1, p}_0\left(\mathcal{W}\right) .$ Then there exist a constant $\mathcal{S}$ such that
\begin{equation*}
\Vert v \Vert_{L^{p^*}(\mathcal{W})} \leq \mathcal{S}\Vert \nabla v \Vert_{L^{p}(\mathcal{W})},
\end{equation*}
with $p^*=\frac{Np}{N-p}.$
\end{lemma}
\begin{lemma}(Hardy's inequality).
 Suppose $p \in (1,N)$ and $u \in W_0^{1, p}\left(\mathcal{W}\right).$  Then we have
\begin{equation*}
\int_{\mathcal{W}} \frac{\vert v\vert^{p}}{\vert x\vert^{p}}\; dx \leq \mathcal{H} \int_{\mathcal{W}}\vert\nabla v\vert^{p}\;dx,
\end{equation*}
with $\mathcal{H}=\left(\frac{p}{N-p}\right)^{p}$ optimal and not achieved constant.
\end{lemma}
\begin{proof}
See \cite{40}.
\end{proof}
An important lemma which  generalize a result of \cite{38} is the following:
\begin{lemma}\label{lapes}
 Assume that $\mathcal{W}$ is an open subset of $\mathbb{R}^{N}$ with finite measure and that 
 $p\in (1,N)$. Let $\psi$ be a measurable function satisfying $\mathcal{T}_{k}(\psi) \in W_{0}^{1, \upsilon}(\mathcal{W})$, for every positive $k$, and such that for some constants $M$ and $L$ we have the inequality
\begin{equation*}
\ds\int_{\mathcal{W}}\vert\nabla \mathcal{T}_{k}(\psi)\vert^{\upsilon}\;dx \leq  Mk^\tau+L, \quad \forall k>0,
\end{equation*}
where $(\tau,\upsilon)\in(0,N)^2$ are given constants. Then $\ds\vert \psi\vert^{\upsilon-\tau}$ belongs to $\ds L^{\frac{N}{N-\upsilon}, \infty}(\mathcal{W}),$ $\vert\nabla \psi\vert^{\upsilon-\tau}$ belongs to $L^{\frac{N}{N-\tau}, \infty}(\mathcal{W})$ and
there exists a constant $C$  which depending only on $N$ and $p$ such that
\begin{equation}\label{es1}
\ds\left\Vert \vert \psi \vert^{\upsilon-\tau}\right\Vert _{L^{\frac{N}{N-\upsilon},\infty}(\mathcal{W})}\leq C(N,p)\left[M+meas(\mathcal{W})^{1-\frac{\upsilon}{\tau}}L^{\frac{\upsilon-\tau}{\upsilon}} \right],
\end{equation} 
\begin{equation}\label{es2}
\ds\left\Vert  \vert \nabla \psi \vert^{\upsilon-\tau}\right\Vert _{L^{\frac{N}{N-\tau},\infty}(\mathcal{W})}\leq C(N,p)\left[M+meas(\mathcal{W})^{\frac{\tau}{\upsilon}(1-\frac{\upsilon}{\tau})}L^{\frac{\upsilon-\tau}{\upsilon}}  \right].
\end{equation} 
\end{lemma}
\begin{proof}
See \cite{16}.
\end{proof}
Consider a nonlinear elliptic problem which can be written as
\begin{equation}\label{problem}
\begin{cases}-\operatorname{div}(b(x,v,\nabla v)+\mathcal{B}(x,v))=\gamma \ds\frac{\vert v\vert^{s-1}v}{\vert x\vert^{p}}+f(x) & \text { in } \mathcal{W},  \\ u=0 & \text { on } \partial \mathcal{W},\end{cases}
\end{equation}
with  $\gamma$ and $s$ are positive constants. Throughout the paper, we assume that the following assumptions hold
true:

$b : \mathcal{W}\times\R\times\R^N \rightarrow \R^N$  is a Carath\'eodory function which satisfies  assumptions:
\begin{equation}\label{b3}
b(x,\eta, \xi).\xi \geq \alpha\varrho^{p-1}(\eta)\vert \xi\vert ^p,
\end{equation}
\begin{equation}\label{b4}
\vert b(x,\eta,\xi)\vert  \leq \mathcal{C} (G(x)+\vert \eta\vert ^{p-1}+\vert \xi\vert ^{p-1}),
\end{equation}
\begin{equation}\label{b5}
 [ b(x,\eta,\xi)-b(x,\eta,\xi')][\xi-\xi'] > 0,
\end{equation}
for almost every $x \in \mathcal{W}$, for every $(\eta,\xi) \in  \R \times\mathbb{R}^N$, $\alpha$  and $\mathcal{C}$ are positive real number, and $G$ is a  nonnegative function in $ L^{p^\prime}(\mathcal{W}).$

$
\mathcal{B} : \mathcal{W}\times\R \rightarrow \R^N \text{ is Carath\'eodory function satisfies the growth condition}
$
\begin{equation}\label{b7}
\vert \mathcal{B}(x,\eta)\vert  \leq c_0(x)\vert \eta\vert ^\lambda,
\end{equation}
 with $0\leq \lambda< (1-\theta)(p-1)$ and  $c_0(x) \in L^{\frac{N}{p-1}}(\mathcal{W}).$
\begin{equation}\label{b8}
 0\leq s< \frac{p(1-\theta )(p-1)}{p^*},\;\gamma\geq 0 \;\text{and}\;f\in L^m(\mathcal{W}),\;\text{with}\; m\geq 1.
\end{equation}

 We first give the definition of a renormalized solutions to problem \eqref{problem}. Then, we will discuss the existence of a renormalized solutions to problem \eqref{problem}. 

We will now provide the definition of a renormalized solutions to problem $\eqref{problem}$.
\begin{definition}\label{d1}
A function $v:\mathcal{W}\rightarrow \R$ is considered a renormalized solutions to Problem $\eqref{problem}$ if it satisfies the following conditions:
\begin{equation}\label{p0}
 v\text{ is measurable and finite almost everywhere in $\mathcal{W}$,}
\end{equation}
\begin{equation}\label{p1}
\mathcal{T}_{k}(\widetilde{\varrho}(v))\in W_0^{1,p}(\mathcal{W})\,\,\,\, \forall k>0,
\end{equation}
\begin{equation}\label{p2}
\lim_{n\rightarrow +\infty}\frac{1}{n}\displaystyle\int_{\lbrace \vert \widetilde{\varrho}(v)\vert \leq n\rbrace}\varrho( v) b(x,v,\nabla v)\nabla v\;dx=0,
\end{equation}
and if for any $h\in W^{1,\infty}(\R)$ with compact support in $\R$, we have
\begin{equation}\label{vp}
\begin{array}{l}
\displaystyle\int_{\mathcal{W}} b(x,v,\nabla v) \varrho(v)\nabla v h^{\prime}(\widetilde{\varrho}(v)) \varphi\;dx+\displaystyle\int_{\mathcal{W}} b(x,v,\nabla v)  \nabla \varphi h(\widetilde{\varrho}(v))\;dx \\
\quad+\displaystyle\int_{\mathcal{W}} \mathcal{B}(x,v) \varrho(v) \nabla v h^{\prime}(\widetilde{\varrho}(v)) \varphi\;dx+\int_{\mathcal{W}} \mathcal{B}(x,v) \nabla \varphi h(\widetilde{\varrho}(v)) \;dx\\
\quad
=\gamma \displaystyle\int_{\mathcal{W}} \frac{\vert v\vert ^{s-1}v}{\vert x\vert ^p}h(\widetilde{\varrho}(v))\varphi \;dx+  \int_{\mathcal{W}} f h(\widetilde{\varrho}(v)) \varphi \;dx\quad
\end{array}
\end{equation}
for every  $\varphi \in W_0^{1,p}(\mathcal{W}) \cap L^{\infty}(\mathcal{W})$.
\end{definition}
\begin{remark}\label{rmk2}
 We notice that, since $\widetilde{\varrho}( \pm \infty)= \pm \infty$ which means that the set $\{\vert \widetilde{\varrho}(v)\vert \leq n\}$ may be equivalent to $\left\{\vert u\vert \leq k_n\right\}$ with $k_n=\max\lbrace \widetilde{\varrho}^{-1}(n),\widetilde{\varrho}^{-1}(-n) \rbrace$, then, due to \eqref{p1}  we deduce that the condition \eqref{p2} is well defined.
\end{remark}
\begin{remark}
 It is worth noting that growth assumption $\eqref{b7}$ on $\mathcal{B}$ together with $\eqref{p0}-\eqref{p2}$ allow to prove that any renormalized solutions $u$ verifies
\begin{equation}\label{pp13}
\lim _{n \rightarrow+\infty} \frac{1}{n} \int_{\mathcal{W}}\vert\mathcal{B}(x,v)\vert\vert\nabla \mathcal{T}_{n}(\widetilde{\varrho}(v))\vert d x=0.
\end{equation}
Indeed, \eqref{jj}, \eqref{jjj} and the growth assumption $\eqref{b7}$ imply that
\begin{align*}
\int_{\mathcal{W}}\vert\mathcal{B}(x,v)\vert \vert\nabla \mathcal{T}_{n}(\widetilde{\varrho}(v))\vert \;dx &\leq  \int_{\mathcal{W}} c_0(x)\vert u\vert^{\lambda}\vert\nabla \mathcal{T}_{n}(\widetilde{\varrho}(v))\vert\; dx,\\
&=\int_\mathcal{W} c_0(x)\frac{\vert u \vert^\lambda\left(1+\vert \widetilde{\varrho}(v)\vert\right)^{p-1}}{\left(1+\vert \widetilde{\varrho}(v)\vert\right)^{p-1}}\vert\nabla \mathcal{T}_{n}(\widetilde{\varrho}(v))\vert\; dx,\\&
\leq \widetilde{C}\int_\mathcal{W} c_0(x)\left(1+\vert \widetilde{\varrho}(v)\vert\right)^{p-1}\vert\nabla \mathcal{T}_{n}(\widetilde{\varrho}(v))\vert\;dx,\\&
\leq \widetilde{\mathcal{C}}c_p
 \int_{\mathcal{W}} c_0(x)\vert\nabla \mathcal{T}_{n}(\widetilde{\varrho}(v))\vert d x\\&
+
\widetilde{\mathcal{C}}c_p\int_{\mathcal{W}} c_0(x)\vert \mathcal{T}_{n}(\widetilde{\varrho}(v))\vert^{p-1}\vert\nabla \mathcal{T}_{n}(\widetilde{\varrho}(v))\vert d x,
\end{align*}
where $c_p=\ds\max\left\lbrace 1,2^{p-2}\right\rbrace$.

By  Hölder's and Sobolev's  inequalities  it follows that 
\small{
\begin{equation*}
 \begin{aligned}
 \int_{\mathcal{W}}\vert\mathcal{B}(x,v)\vert \vert\nabla \mathcal{T}_{n}(\widetilde{\varrho}(v))\vert d x &\leq \widetilde{\mathcal{C}}
 \int_{\mathcal{W}} c_0(x)\vert\nabla \mathcal{T}_{n}(\widetilde{\varrho}(v))\vert d x\\&+
\widetilde{\mathcal{C}}c_p\int_{\mathcal{W}} c_0(x)\vert \mathcal{T}_{n}(\widetilde{\varrho}(v))\vert^{p-1}\vert\nabla \mathcal{T}_{n}(\widetilde{\varrho}(v))\vert d x,\\& 
\leq \widetilde{\mathcal{C}}c_p \Vert c_0(x)\Vert_{L^{p^\prime}(\mathcal{W})}\Vert\nabla \mathcal{T}_{n}(\widetilde{\varrho}(v))\Vert_{L^{p}(\mathcal{W})}\\&+
\widetilde{\mathcal{C}}c_p\Vert c_0(x)\Vert_{L^{\frac{N}{p-1}}(\mathcal{W})}\left\Vert \mathcal{T}_{n}(\widetilde{\varrho}(v))\right\Vert_{L^{p^{*}}(\mathcal{W})}^{p-1}\left\Vert\nabla \mathcal{T}_{n}(\widetilde{\varrho}(v))\right\Vert_{L^{p}(\mathcal{W})},\\
&\leq \widetilde{\mathcal{C}}c_p \Vert c_0(x)\Vert_{L^{p^\prime}(\mathcal{W})}\Vert\nabla \mathcal{T}_{n}(\widetilde{\varrho}(v))\Vert_{L^{p}(\mathcal{W})}\\&+\widetilde{\mathcal{C}}c_p\mathcal{S}^{p-1}\Vert c_0(x)\Vert_{L^{\frac{N}{p-1}}(\mathcal{W})}\left\Vert\nabla \mathcal{T}_{n}(\widetilde{\varrho}(v))\right\Vert_{L^{p}(\mathcal{W})}^{p},
\end{aligned}
\end{equation*}}
which, using Young's inequality  and $\eqref{p2}$, gives \eqref{pp13}.
\end{remark}
 \begin{remark}
Note that the term $\displaystyle\gamma\int_{\mathcal{W}} \frac{\vert  v\vert^{s-1}v}{\vert x\vert^p}h(\widetilde{\varrho}(v))\varphi \;dx$ is well-defined. Indeed, let $k>0$ such that $supp(h)\subset [-k,k]$. 
It follows from \eqref{jjj}, Hölder's and Hardy Inequalities that
 \begin{align*}
 \ds\left\vert\gamma\int_{\mathcal{W}} \frac{\vert v\vert^{s-1}v}{\vert x\vert^p} h(\widetilde{\varrho}(v))v \;dx\right\vert&\leq\gamma\int_{\mathcal{W}} \frac{\vert v\vert^s}{(1+\vert \widetilde{\varrho}(v)\vert)^{p-1}}\frac{(1+\vert \widetilde{\varrho}(v)\vert)^{p-1}}{\vert x\vert^p}\vert h(\widetilde{\varrho}(v))\vert \vert \varphi\vert \;dx,
 \\
 &\leq C_\gamma\left[\left(\int_{\mathcal{W}}\frac{dx}{\vert x\vert^p}\;dx\right)^{\frac{1}{p^\prime}}+\mathcal{H}^{\frac{1}{p^\prime}}\left(\int_\mathcal{W}\vert \nabla \mathcal{T}_{k}(\widetilde{\varrho}(v))\vert^p \;dx\right)^{\frac{1}{p^\prime}}\right]
 \end{align*}
 where 
\begin{equation}
C_\gamma=\ds \gamma c_p \widetilde{\mathcal{C}}\Vert h\Vert_{L^\infty(\R)}\Vert \varphi\Vert_{L^\infty(\mathcal{W})}\left(\int_{\mathcal{W}}\frac{dx}{\vert x\vert^p}\;dx \right)^{\frac{1}{p}}.
\end{equation}  
 Then we have, from $\eqref{p1}$, that $\ds\frac{\vert v\vert^{s-1}v}{\vert x\vert^p}h(\widetilde{\varrho}(v))\varphi \in L^{1}(\mathcal{W}).$
 \end{remark}
 \begin{remark}
The renormalized equation \eqref{vp} is formally obtained through a pointwise multiplication of \eqref{problem} by $h(v) \varphi$. Let us observe that by the previous remark, \eqref{p1} and the properties of $h$, every term in \eqref{vp} makes sense.
\end{remark}
\section{Existence of renormalized solutions}\label{s4}
The main result of the present paper is the following existence result.
\begin{theorem}\label{t1}
 Let us assume that the assumptions $\eqref{b3}-\eqref{b8}$ hold and suppose that $\ds f\in L^m(\mathcal{W})$ with 
 \begin{equation}
 \ds 1\leq m<m_\theta=\frac{pN}{pN-(N-p)(\theta(p-1)+1)}.
 \end{equation}
 If $\ds \lambda<(1-\theta)(p-1),$ $\ds s<\frac{p(1-\theta)(p-1)}{p^*}$ and $\ds\gamma\geq0$, then there exists a renormalized solutions
of equation $(\ref{problem})$.
\end{theorem} 
\begin{proof}
The proof of this result follows a classical approach that involves  introducing a sequence of approximate problems. Subsequently, we establish a priori estimates for both the approximate solutions and their gradients in Lorentz-Marcinkiewicz spaces, thereby providing an estimate in $L^1(\mathcal{W})$ for the singular term. Next, we prove an energy estimate, which constitutes a crucial element for the subsequent stages of the proof. Moreover, in the fourth step of the proof, we establish the a.e. convergence of gradient in $\mathcal{W}$ by proving Lemma \ref{t2}. Finally, we pass to the limit in the approximate problem.

$\bullet$ \textbf{First Step: Approximate problem} 

Let's introduce a regularization of the data as follows: for a fixed $\varepsilon > 0, $ let's define
\begin{equation}\label{aapr}
b_\varepsilon(x,\eta,\xi)=b(x,\mathcal{T}_{\frac{1}{\varepsilon}}(\eta),\xi)\,\,\,\,\forall \eta\in \R,
\end{equation}
\begin{equation}\label{phap}
\mathcal{B}_\varepsilon(x,\eta)=\mathcal{B}(x,\mathcal{T}_{\frac{1}{\varepsilon}}(\eta))\,\,\,\,\forall \eta\in \R,
\end{equation}
\begin{equation}\label{cvdf}
f^\varepsilon=\mathcal{T}_{\frac{1}{\varepsilon}}(f)\;\;\; \text{and}\;f^\varepsilon \rightarrow f  \;\;\text{strongly in}\; L^m(\mathcal{W}).
\end{equation}
Observe that
\begin{equation}\label{m1m}
\begin{aligned}
\ds\vert b_\varepsilon(x,\eta,\xi)\vert&=\vert b(x,\mathcal{T}_{\frac{1}{\varepsilon}}(\eta),\xi)\vert\\
&\leq \mathcal{C}(G(x)+\vert \mathcal{T}_{\frac{1}{\varepsilon}}(\eta) \vert^{p-1}+ \vert \xi\vert^{p-1}),\\
&\leq \mathcal{C}(G(x)+\frac{1}{\varepsilon^{p-1} }+ \vert \xi\vert^{p-1})\in L^{p^\prime}(\mathcal{W}),
\end{aligned}
\end{equation}
and
\begin{equation}\label{m2m}
\ds\begin{aligned}
 b_\varepsilon(x,\eta,\xi)\xi=b(x,\mathcal{T}_{\frac{1}{\varepsilon}}(\eta),\xi)\xi
\geq \frac{\alpha\vert \xi \vert^p}{(1+\frac{1}{\varepsilon})^{\theta(p-1)}}
\geq \widetilde{\alpha} \vert \xi \vert^p,
\end{aligned}
\end{equation}
moreover
\begin{equation}\label{m3m}
\begin{aligned}
\ds\vert \mathcal{B}_\varepsilon(x,\eta)\vert=\vert\mathcal{B}(x,\mathcal{T}_{\frac{1}{\varepsilon}}(\eta))\vert\leq \frac{c_0(x)}{\varepsilon^\lambda}\in L^{p^\prime}(\mathcal{W}).
\end{aligned}
\end{equation}
Let $v_\varepsilon \in W_0^{1,p}(\mathcal{W})$ be a weak solution to the following approximate problem 
\small{
\begin{equation}\label{problemapr}
\begin{aligned}
\begin{cases}
-div(b_\varepsilon(x,v_\varepsilon,\nabla v_\varepsilon)+\mathcal{B}_\varepsilon(x,v_\varepsilon))=\gamma \displaystyle\frac{\vert \mathcal{T}_{\frac{1}{\varepsilon}}(v_\varepsilon)\vert ^{s-1}\mathcal{T}_{\frac{1}{\varepsilon}}(v_\varepsilon)}{\vert x\vert ^p+\varepsilon}+f^\varepsilon & \text{in}\, \mathcal{W}, \\
v_\varepsilon=0 &\text{on} \, \partial \mathcal{W},
\end{cases}
\end{aligned}
\end{equation}}
in the sens that
\begin{equation}\label{vpap}
\begin{aligned}
\displaystyle\int_{\mathcal{W}} b_\varepsilon(x, v_\varepsilon, \nabla v_\varepsilon)  \nabla \varphi \,dx 
+\displaystyle\int_{\mathcal{W}} \mathcal{B}_\varepsilon(x, v_\varepsilon)  \nabla \varphi \,dx 
&=\gamma \displaystyle\int_{\mathcal{W}} \frac{\vert \mathcal{T}_{\frac{1}{\varepsilon}}(v_\varepsilon)\vert ^{s-1}\mathcal{T}_{\frac{1}{\varepsilon}}(v_\varepsilon)}{\vert x\vert ^p+\varepsilon}\varphi \,dx\\& + \int_{\mathcal{W}} f^\varepsilon  \varphi\,dx
\end{aligned}
\end{equation}
$\forall \varphi\in W_0^{1,p}(\mathcal{W})\cap L^{\infty}(\mathcal{W}).$

According to \eqref{m1m}, \eqref{m2m} and \eqref{m3m} the existence of a solution $v_\varepsilon$ of $(\ref{problemapr})$ is a well-known result $($see, e.g.,\cite{2}$).$

$\bullet$ \textbf{Second step: A priori estimates} 

In this step we deal with the approximate problem \eqref{problemapr}. We begin by proving some a priori estimates on  $\widetilde{\varrho}(v_\varepsilon),$ $ \nabla \widetilde{\varrho}(v_\varepsilon),$ $v_\varepsilon,$ and $\nabla v_\varepsilon$.
\begin{proposition}\label{pr1}
Suppose that $ f\in L^m(\mathcal{W})$ with 
$ 1\leq m<m_\theta.$
Under the assumptions $\eqref{b3}-\eqref{b7}$ and 
 if $ \lambda<(1-\theta)(p-1),$ $ s<\frac{p(1-\theta)(p-1)}{p^*}$ and $\gamma\geq0$. Then, every weak solution of the problem $\eqref{vpap}$ satisfies
\begin{equation}\label{esta2}
\ds\left\Vert  \vert \widetilde{\varrho}(v_\varepsilon) \vert ^{p-1} \right\Vert _{L^{\frac{N}{N-p},\infty}(\mathcal{W})}\leq c_1\left(N,p,\alpha,meas(\mathcal{W}),c_0\right)
\end{equation}
\begin{equation}\label{esta1}
\ds\left\Vert  \vert \nabla \widetilde{\varrho}(v_\varepsilon) \vert ^{p-1} \right\Vert _{L^{\frac{N}{N-1},\infty}(\mathcal{W})}\leq c_2\left(N,p,\alpha,meas(\mathcal{W}),c_0\right)
\end{equation}
\begin{equation}\label{est2}
\ds\left\Vert  \vert  v_\varepsilon \vert ^{(1-\theta)(p-1)} \right\Vert _{L^{\frac{N}{N-p},\infty}(\mathcal{W})}\leq c_3\left(N,p,\alpha,meas(\mathcal{W}),c_0\right)
\end{equation}
\begin{equation}\label{est1}
\ds\left\Vert  \vert \nabla v_\varepsilon \vert ^{(1-\theta)(p-1)} \right\Vert _{L^{\frac{N}{N-1-\theta(p-1)},\infty}(\mathcal{W})}\leq c_4\left(N,p,\alpha,meas(\mathcal{W}),c_0\right)
\end{equation}
for some positive constants $c_1,\;c_2,\;c_3$ and $c_4$.
\end{proposition}
\begin{proof}
We will divide the proof of this proposition into two steps. Initially, we will establish that our sequence of weak solutions $v_\varepsilon$ is almost everywhere finite in $\mathcal{W}$. Following this outcome, we will derive uniform bounds for $\mathcal{T}_{k}(\widetilde{\varrho}(v_\varepsilon))$ and $\mathcal{T}_{k}(v_\varepsilon)$ in Lebesgue spaces.  Let us mention that throughout the paper $C_i,\; i\in \N^*,$ denotes positive constants independent of $\varepsilon$ that are different from line to line. At last, for any measurable set $D\subset \R^N$, $D^c$ denotes its complement.

\textit{$\bullet$ Step 1: $v_\varepsilon$ is finite a.e. in $\mathcal{W}$.}

 This step is devoted to establish that assuming conditions $(\ref{b3}),$ $(\ref{b7}),$ and $f\in L^1(\mathcal{W})$, the sequence of weak solution satisfies 
\begin{equation}\label{aee}
 \forall \eta >0, \exists k_\eta>0 \;\;\; meas\lbrace\vert v_\varepsilon \vert> k_\eta \rbrace\leq \frac{1}{\eta^p},\quad \text{uniformaly w.r.t  $\varepsilon$.}
\end{equation}
 To this aim, let us consider the real valued function $\psi_p: \R\rightarrow \R$  defined by
\begin{equation*}
\psi_p(t)=\int_0^t \frac{dr}{(\beta_\varepsilon ^{\frac{1}{p-1}}+\vert r\vert )^p},
\end{equation*}
where $\beta_\varepsilon>1,$ is a suitably chosen parameter. 

We use $\psi_p(\widetilde{\varrho}(v_\varepsilon))$ as
test function in $(\ref{vpap})$, we get
\begin{equation}\label{vpe}
\begin{aligned}
\ds\int_{\mathcal{W}}&b_\varepsilon(x,v_\varepsilon,\nabla v_\varepsilon)  \nabla \psi_p(\widetilde{\varrho}(v_\varepsilon))+\int_{\mathcal{W}} \mathcal{B}_\varepsilon(x, v_\varepsilon)  \nabla \psi_p(\widetilde{\varrho}(v_\varepsilon))\\&
= \gamma \int_{\mathcal{W}} \displaystyle\frac{\left\vert \mathcal{T}_{\frac{1}{\varepsilon}}(v_\varepsilon)\right\vert ^{s-1}\mathcal{T}_{\frac{1}{\varepsilon}}(v_\varepsilon)}{\vert x\vert ^p+\varepsilon}\psi_p(\widetilde{\varrho}(v_\varepsilon))+\int_{\mathcal{W}} f^\varepsilon \psi_p(\widetilde{\varrho}(v_\varepsilon)) 
.
\end{aligned}
\end{equation}
Using the assumptions $\eqref{b3}$ and $\eqref{b7}$, we obtain
\begin{align*}
 \displaystyle\alpha\int_{\mathcal{W}}   \frac{\vert \nabla \widetilde{\varrho}(v_\varepsilon)\vert ^p}{(\beta_\varepsilon ^{\frac{1}{p-1}}+\vert \widetilde{\varrho}(v_\varepsilon)\vert )^p}\;dx
&\leq 
\ds\displaystyle\int_{\mathcal{W}}   c_0(x)\ds\frac{\vert  v_\varepsilon\vert ^\lambda}{(\beta_\varepsilon ^{\frac{1}{p-1}}+\vert \widetilde{\varrho}(v_\varepsilon)\vert )^{p-1}} \ds\frac{\vert \nabla \widetilde{\varrho}(v_\varepsilon)\vert }{(\beta_\varepsilon ^{\frac{1}{p-1}}+\vert \widetilde{\varrho}(v_\varepsilon)\vert )}\;dx\\&
 +\ds\frac{1}{\beta_\varepsilon (p-1)}M_\varepsilon,
\end{align*}
where  
\begin{equation*}
M_\varepsilon=\gamma  \displaystyle\int_\mathcal{W} \dfrac{\left\vert \mathcal{T}_{\frac{1}{\varepsilon}}( v_\varepsilon)\right\vert ^s\;dx}{\vert x\vert ^p+\varepsilon}+\big\Vert  f^\varepsilon \big\Vert _{{L}^1(\mathcal{W})}.
\end{equation*}
 Thanks to \eqref{jjj}, for any $\lambda<(1-\theta )(p-1)$, we have
 \begin{equation*}
 \ds\frac{\vert  v_\varepsilon\vert ^\lambda}{(\beta_\varepsilon ^{\frac{1}{p-1}}+\vert \widetilde{\varrho}(v_\varepsilon)\vert )^{p-1}}\leq \widetilde{C}.
 \end{equation*}
Thus, applying Young's inequality, we obtain
\begin{align*}
 \ds\int_{\mathcal{W}}   \frac{\vert \nabla \widetilde{\varrho}(v_\varepsilon)\vert ^p}{(\beta_\varepsilon ^{\frac{1}{p-1}}+\vert \widetilde{\varrho}(v_\varepsilon)\vert )^p}\;dx
&\leq  C_1+\frac{p^\prime}{\alpha\beta_\varepsilon (p-1)}M_\varepsilon.
\end{align*}
Now, if we choose  $\beta_\varepsilon =1+\ds\frac{p^\prime}{\alpha(p-1)}M_\varepsilon,$  we deduce that
\begin{align*}
 \displaystyle\int_{\mathcal{W}}   \frac{\vert \nabla \widetilde{\varrho}(v_\varepsilon)\vert ^p}{(\beta_\varepsilon ^{\frac{1}{p-1}}+\vert \widetilde{\varrho}(v_\varepsilon)\vert )^p}\;dx
&\leq  C_1+1,
\end{align*} 
 which implies, applying Poincaré's inequality and for any $h>0,$ that 
\begin{align*}
meas\left\lbrace\vert \widetilde{\varrho}(v_\varepsilon)\vert > h\beta_\varepsilon ^{\frac{1}{p-1}}\right\rbrace&=\frac{1}{[ln(1+h)]^p}\displaystyle\int_{\left\lbrace\vert \widetilde{\varrho}(v_\varepsilon)\vert > h\beta_\varepsilon ^{\frac{1}{p-1}}\right\rbrace} \left[ln(1+h)\right]^p\;dx,\\
&\leq \frac{1}{[ln(1+h)]^p}\displaystyle\int_{\left\lbrace\vert \widetilde{\varrho}(v_\varepsilon)\vert > h\beta_\varepsilon ^{\frac{1}{p-1}}\right\rbrace} \left[ln\left(1+\frac{\vert \widetilde{\varrho}(v_\varepsilon)\vert }{\beta_\varepsilon ^{\frac{1}{p-1}}}\right)\right]^p\;dx,\\
&\leq \frac{1}{[ln(1+h)]^p}\displaystyle\int_{\mathcal{W}} \left[ln\left(1+\frac{\vert \widetilde{\varrho}(v_\varepsilon)\vert }{\beta_\varepsilon ^{\frac{1}{p-1}}}\right)\right]^p \;dx,\\
&\leq \frac{A}{[ln(1+h)]^p},
\end{align*}
where 
\begin{equation*}\label{A}
A=c(N,p)\left(1+C_1\right).
\end{equation*}
Then, for any $\eta > 0,$ we have
\begin{equation}\label{ae}
\ds meas \left\lbrace \vert \widetilde{\varrho}(v_\varepsilon)\vert > \sigma_\eta(\varepsilon) \right\rbrace\leq \ds\frac{1}{\eta^p},
\end{equation}
where 
\begin{equation}\label{sigmaa}
\ds\sigma_\eta(\varepsilon)=\left(\exp\left(\eta A^{\frac{1}{p}}\right)-1\right)\beta_\varepsilon ^{\frac{1}{p-1}}.
\end{equation}
\begin{remark}
Notice that, from \eqref{sigmaa} and recalling the definition of $\beta_\varepsilon$, we can express $\sigma_\eta$ as follows
\begin{equation*}
\sigma_\eta(\varepsilon)=\left(\exp\left(\eta A^{\frac{1}{p}}\right)-1\right)\left(1+\ds\frac{p^\prime}{\alpha(p-1)}M_\varepsilon \right)^{\frac{1}{p-1}}.
\end{equation*}
Therefore, it is crucial to emphasize the need to establish the boundedness of the term $M_\varepsilon$ uniformly with respect to $\varepsilon$ to prove that $\widetilde{\varrho}(v_\varepsilon)$ is finite almost everywhere in $\mathcal{W}$.
Referring back to the definition of $M_\varepsilon$, we can deduce the following inequality
\begin{equation*}
M_\varepsilon\leq \gamma \left\Vert \frac{\left\vert \mathcal{T}_{\frac{1}{\varepsilon}}(v_\varepsilon)\right\vert ^{s}}{\vert x\vert ^p} \right\Vert_{L^1(\mathcal{W})}+\left\Vert  f^\varepsilon \right\Vert _{L^1(\mathcal{W})}.
\end{equation*}
To derive the desired estimate, our task is to prove the boundedness of the Hardy potential term in $L^1(\mathcal{W})$. To this aim, we adopt the approach outlined in \cite{22}.
\end{remark}
 For any given $k \geq 0$, let  $\mathcal{T}_{k}(\widetilde{\varrho}(v_\varepsilon))$ be chosen as a test function in \eqref{vpap}. This yields 
\begin{align*}
\ds \int_\mathcal{W} & b_\varepsilon(x,v_\varepsilon,\nabla v_\varepsilon) \nabla \mathcal{T}_{k}(\widetilde{\varrho}( v_\varepsilon))\;dx+\displaystyle \int_\mathcal{W} \mathcal{B}_\varepsilon(x,v_\varepsilon) \nabla \mathcal{T}_{k}(\widetilde{\varrho}( v_\varepsilon))\;dx\\&= \gamma \displaystyle\int_{\mathcal{W}} \frac{\vert \mathcal{T}_{\frac{1}{\varepsilon}}(v_\varepsilon)\vert ^{s-1}\mathcal{T}_{\frac{1}{\varepsilon}}(v_\varepsilon)}{\vert x\vert ^p+\varepsilon}\mathcal{T}_{k}(\widetilde{\varrho}( v_\varepsilon))\;dx+\int_{\mathcal{W}} f^\varepsilon  \mathcal{T}_{k}(\widetilde{\varrho}( v_\varepsilon))\;dx.  
\end{align*}
Using  $\eqref{b3}$ and  $\eqref{b7}$,  we get
\begin{align*}
\alpha \displaystyle \int_\mathcal{W} \big\vert  \nabla \mathcal{T}_{k}(\widetilde{\varrho}( v_\varepsilon)) \big\vert ^p \;dx & 
\leq
\displaystyle \int_{\mathcal{W}} c_0(x)\vert v_\varepsilon\vert ^\lambda\vert \nabla \mathcal{T}_{k}(\widetilde{\varrho}( v_\varepsilon))\vert \; dx+kM_\varepsilon,\\
&\leq \ds\int_{\mathcal{W}} c_0(x)\frac{\vert v_\varepsilon \vert^\lambda \vert\nabla \mathcal{T}_{k}(\widetilde{\varrho}( v_\varepsilon))\vert}{\left(1+\vert \widetilde{\varrho}( v_\varepsilon)\vert\right)^{p-1}}\left(1+\vert \widetilde{\varrho}( v_\varepsilon)\vert\right)^{p-1} \;dx\\&+k M_\varepsilon,
\end{align*}
Therefore, through \eqref{jjj} and employing Hölder's, Young's and Sobolev's inequalities, we obtain 
\begin{align*}
\alpha \displaystyle \int_\mathcal{W} \big\vert  \nabla \mathcal{T}_{k}(\widetilde{\varrho}( v_\varepsilon)) \big\vert ^p \;dx & 
\leq \widetilde{C}\ds\int_{\mathcal{W}} c_0(x)\vert \widetilde{\varrho}( v_\varepsilon)\vert^{p-1} \vert\nabla \mathcal{T}_{k}(\widetilde{\varrho}( v_\varepsilon))\vert\;dx\\&+
\widetilde{C}\ds\int_{\mathcal{W}} c_0(x) \vert\nabla \mathcal{T}_{k}(\widetilde{\varrho}( v_\varepsilon))\vert\;dx+kM_\varepsilon,\\
&\leq  \widetilde{C}\displaystyle \int_{Z_{\eta,\varepsilon}^c} c_0(x)\vert \widetilde{\varrho}( v_\varepsilon)\vert ^{p-1}\vert \nabla \mathcal{T}_{k}(\widetilde{\varrho}( v_\varepsilon))\vert \; dx\\&+ \widetilde{C}\displaystyle \int_{Z_{\eta,\varepsilon}} c_0(x)\vert \widetilde{\varrho}( v_\varepsilon)\vert ^{p-1}\vert \mathcal{T}_{k}(\widetilde{\varrho}( v_\varepsilon))\vert \; dx\\&
+C_2+\frac{\alpha}{2p}\displaystyle \int_{\mathcal{W}} \vert \nabla \mathcal{T}_{k}(\widetilde{\varrho}( v_\varepsilon))\vert ^p\; dx+kM_\varepsilon,\\
&\leq \widetilde{C} \mathcal{S}^{p-1}\big\Vert  c_0(x)\big\Vert _{L^{\frac{N}{p-1}}(Z_{\eta,\varepsilon}^c)}\int_{\mathcal{W}} \vert \nabla \mathcal{T}_{k}(\widetilde{\varrho}( v_\varepsilon))\vert ^p\; dx \\
&+\frac{\alpha}{p}\int_{\mathcal{W}} \vert \nabla \mathcal{T}_{k}(\widetilde{\varrho}( v_\varepsilon))\vert ^p\; dx+C_\varepsilon+kM_\varepsilon,
\end{align*}
where 
\begin{equation*}
Z_{\eta,\varepsilon}=\lbrace x\in \mathcal{W},\vert \widetilde{\varrho}( v_\varepsilon(x))\vert \leq \sigma_\eta(\varepsilon) \rbrace,
\end{equation*}
and
\begin{equation*}
 C_\varepsilon=C_2+\left(\frac{2p}{\alpha} \right)^{\frac{p^\prime}{p}}\frac{\left(\widetilde{C}\sigma_\eta^{p-1}(\varepsilon)\right)^{p^\prime}}{p^\prime}\int_\mathcal{W} \vert c_0(x)\vert^{p^\prime}\;dx,
\end{equation*}
i.e.,
\begin{align*}
\frac{\alpha}{p^\prime} \displaystyle \int_\mathcal{W} \big\vert  \nabla \mathcal{T}_{k}(\widetilde{\varrho}( v_\varepsilon)) \big\vert ^p \;dx & 
\leq \widetilde{C} \mathcal{S}^{p-1}\big\Vert  c_0(x)\big\Vert _{L^{\frac{N}{p-1}}(Z_{\eta,\varepsilon}^c)}\int_{\mathcal{W}} \vert \nabla \mathcal{T}_{k}(\widetilde{\varrho}( v_\varepsilon))\vert ^p\; dx\\&+C_\varepsilon+kM_\varepsilon.
\end{align*}
Note that, for every  $\varepsilon>0$, we can select $\eta=\bar{\eta},$ in \eqref{sigmaa}, such that
\begin{equation*}
\ds\frac{p^{\prime}}{\alpha}\widetilde{C}\mathcal{S}^{p-1}\left\Vert  c_0(x)\right\Vert _{ L^{\frac{N}{p-1}}\left(Z_{\bar{\eta},\varepsilon}^c\right)}\leq \frac{1}{2}.
\end{equation*}
As a result, we derive that
\begin{align*}
\displaystyle \int_\mathcal{W} \big\vert  \nabla \mathcal{T}_{k}(\widetilde{\varrho}( v_\varepsilon)) \big\vert ^p \;dx & 
\leq L_\varepsilon +M_\varepsilon^\prime k,
\end{align*}
where
\begin{equation*}
L_\varepsilon=\frac{2 p^\prime}{\alpha}C_\varepsilon  \text{ and } M_\varepsilon^\prime=\frac{2 p^\prime}{\alpha}M_\varepsilon.
\end{equation*}
By  lemma \ref{lapes}, we get
\begin{equation*}
\left\Vert  \left\vert  \widetilde{\varrho}( v_\varepsilon) \right\vert ^{p-1}\right\Vert _{L^{\frac{N}{N-p},\infty}(\mathcal{W})}\leq C(N,p)\left[M^\prime_\varepsilon+meas(\mathcal{W})^{\frac{1}{p}}L_\varepsilon^{\frac{1}{p^\prime}} \right].
\end{equation*}
 Moreover, using the fact that $\widetilde{\varrho}( v_\varepsilon)$ exhibits behavior similar to $\vert v_\varepsilon \vert^{1-\theta }$ for any $\theta <1$, we  conclude that
\begin{equation*}
\left\Vert  \left\vert  v_\varepsilon \right\vert ^{(1-\theta )(p-1)}\right\Vert _{L^{\frac{N}{N-p},\infty}(\mathcal{W})}\leq C(N,p)\left[M^\prime_\varepsilon+meas(\mathcal{W})^{\frac{1}{p}}L_\varepsilon^{\frac{1}{p^\prime}} \right].
\end{equation*} 
Given $ s<\frac{p(1-\theta)(p-1)}{p^*}$, and utilizing $\eqref{ceb}$ along with the application of Hölder's inequality (with exponents $\rho=\frac{(1-\theta)(p-1)}{s}$ and $\rho^\prime=\frac{\rho}{\rho-1}$), we obtain
  \begin{align*}  
  \left\Vert \frac{\left\vert \mathcal{T}_{\frac{1}{\varepsilon}}(v_\varepsilon)\right\vert ^{s}}{\vert x\vert ^p} \right\Vert_{L^1(\mathcal{W})}&\leq \left(\displaystyle\int_\mathcal{W} \vert  v_\varepsilon \vert ^{(1-\theta )(p-1)}\;dx\right)^{\frac{1}{\rho}}\left(\displaystyle\int_\mathcal{W} \frac{dx}{\vert  x \vert ^{p\rho^\prime}}\right)^{\frac{1}{\rho^\prime}},\\
  &\leq C_3 \left\Vert  \left\vert  v_\varepsilon \right\vert ^{(1-\theta )(p-1)}\right\Vert _{L^{1}(\mathcal{W})}^{\frac{1}{\rho}},\\
  &\leq C_3\left\Vert  \left\vert  v_\varepsilon \right\vert ^{(1-\theta )(p-1)}\right\Vert _{L^{\frac{N}{N-p},\infty}(\mathcal{W})}^{\frac{1}{\rho}},\\
  & \leq C_3 \left[ C(N,p)\left[M^\prime_\varepsilon+meas(\mathcal{W})^{\frac{1}{p}}L_\varepsilon^{\frac{1}{p^\prime}} \right] \right]^{\frac{s}{p-1}},\\
  &\leq  C_4 M_\varepsilon^{\frac{s}{p-1}}+
  C_5L_\varepsilon^{\frac{s}{p}}.
  \end{align*} 
  \begin{remark}
Let us emphasize that we don't know if the value $\frac{p(p-1)(1-\theta)}{p^*}$ is optimal in order to obtain the above estimate. Thus, it would be very interesting to know what happen in the
lacking set $\frac{p(p-1)(1-\theta)}{p^*} \leq s< (p-1)(1-\theta).$ This delicate case, will be dealt with in an upcoming paper, wherein the focus lies on exploring \textit{renormalized solutionss} with a new approach. 
  \end{remark}
  After easy calculations, we prove that
  \begin{equation*}
  L_\varepsilon=\frac{2 p^\prime}{\alpha}\left(C_2+\left(\frac{2p}{\alpha} \right)^{\frac{p^\prime}{p}}\frac{\left(\widetilde{C}\sigma_\eta^{p-1}(\varepsilon)\right)^{p^\prime}}{p^\prime}\int_\mathcal{W} \vert c_0(x)\vert^{p^\prime}\;dx \right) \leq C_6+C_7 M_\varepsilon^{p^\prime}.
\end{equation*}
Which implies, by applying Young's inequality, that
     \begin{align*}  
  \left\Vert \frac{\left\vert \mathcal{T}_{\frac{1}{\varepsilon}}(v_\varepsilon)\right\vert ^{s}}{\vert x\vert ^p} \right\Vert_{L^1(\mathcal{W})}&\leq C_9+ C_8 M_\varepsilon^{\frac{s}{p-1}},\\   
 &\leq C_{10} +\frac{1}{2\gamma} M_\varepsilon,
  \end{align*} 
thus, we deduce that  
   \begin{equation}\label{bn}
   \left\Vert \frac{\left\vert \mathcal{T}_{\frac{1}{\varepsilon}}(v_\varepsilon)\right\vert ^{s}}{\vert x\vert ^p} \right\Vert_{L^1(\mathcal{W})}\leq C_{11}.
   \end{equation}   
 So that there exists a positive constant $\bar{\sigma}_\eta$, independent of $\varepsilon$, such that
 \begin{align*}
\sigma_\eta(\varepsilon)\leq \left(\exp\left(\eta A^{\frac{1}{p}}\right)-1\right)\left(1+\ds\frac{p^\prime}{\alpha(p-1)}\left(\gamma C_{11}+\Vert f\Vert_{L^1(\mathcal{W})} \right) \right)^{\frac{1}{p-1}}= \bar{\sigma}_\eta.
\end{align*}
Observe that, by the definition of $\bar{\sigma}_\eta$, it results 
\begin{equation*}
\lim_{\eta\rightarrow +\infty}\bar{\sigma}_\eta=+\infty.
\end{equation*}
 Moreover,  since $meas \left\lbrace \vert \widetilde{\varrho}(v_\varepsilon)\vert > \bar{\sigma}_\eta \right\rbrace \leq meas \left\lbrace \vert \widetilde{\varrho}(v_\varepsilon)\vert > \sigma_\eta(\varepsilon) \right\rbrace,$
and recalling \eqref{ae}, we affirm that $\widetilde{\varrho}(v_\varepsilon)$ is finite a.e. in $\mathcal{W}.$ 
Therefore, since $\widetilde{\varrho}$ is $C^1$ non decreasing function, there exists $k_\eta>0$ such that
\begin{equation*}
\ds meas \left\lbrace \vert v_\varepsilon\vert > k_\eta \right\rbrace\leq \frac{1}{\eta^p}.
\end{equation*}
This implies, since $\widetilde{\varrho}(\pm \infty)=\pm \infty$, that $v_\varepsilon$ is also finite a.e. in $\mathcal{W}$. Which is equivalent to \eqref{aee}.

\textit{$\bullet$ Step 2: $\mathcal{T}_{k}(\widetilde{\varrho}( v_\varepsilon))$ and $\mathcal{T}_{k}(v_\varepsilon)$  are bounded in $W_0^{1,p}(\mathcal{W})$.}
 
 Choosing $\mathcal{T}_{k}(\widetilde{\varrho}( v_\varepsilon))$ as test function in \eqref{vpap}, and following the same reasoning as before, we obtain
   \begin{equation}\label{alphauepsilon}
\displaystyle \int_\mathcal{W} \left\vert  \nabla \mathcal{T}_{k}(\widetilde{\varrho}( v_\varepsilon)) \right\vert ^p \;dx \leq L+Mk \quad\quad \forall k>0,
\end{equation}
where
\begin{equation*}
 L=C_6+C_8\left(\gamma C_{11}+\Vert  f \Vert _{L^1(\mathcal{W})}\right) \text{ and }M=\ds\frac{2 p^\prime}{\alpha}\left(\gamma C_{11}+\Vert  f \Vert _{L^1(\mathcal{W})}\right).
\end{equation*}
\begin{remark}
We note that, in this step, we have employed the measurable set
\begin{equation}\label{ZZ}
\bar{Z}_{\eta,\varepsilon}=\lbrace x\in \mathcal{W},\vert \widetilde{\varrho}( v_\varepsilon(x))\vert \leq \bar{\sigma}_\eta \rbrace,
\end{equation}
instead of the set $Z_{\eta,\varepsilon}.$
\end{remark}
The  estimate \eqref{alphauepsilon}, together with the lemma \ref{lapes}, enables us to deduce \eqref{esta2} and \eqref{esta1}.

Taking $\mathcal{T}_{k}(v_\varepsilon)$ as test function in \eqref{vpap}, we obtain
\begin{align*}
\displaystyle \int_\mathcal{W} & b_\varepsilon(x,v_\varepsilon,\nabla v_\varepsilon) \nabla \mathcal{T}_{k}(v_\varepsilon)\;dx+\displaystyle \int_\mathcal{W} \mathcal{B}_\varepsilon(x,v_\varepsilon) \nabla \mathcal{T}_{k}(v_\varepsilon)\;dx\\&= \gamma \displaystyle\int_{\mathcal{W}} \frac{\vert \mathcal{T}_{\frac{1}{\varepsilon}}(v_\varepsilon)\vert ^{s-1}\mathcal{T}_{\frac{1}{\varepsilon}}(v_\varepsilon)}{\vert x\vert ^p+\varepsilon}\mathcal{T}_{k}(v_\varepsilon)\;dx
+\int_{\mathcal{W}} f^\varepsilon  \mathcal{T}_{k}(v_\varepsilon)\;dx.
\end{align*}
By $\eqref{b3},$ $\eqref{b7},$ and utilizing Hölder's, Young's, and Sobolev's inequalities, we get
\begin{align*}
\frac{\alpha}{p^\prime} I_k 
 \leq C_{12}\Vert c_0(x)\Vert_{L^{\frac{N}{p-1}}(\mathcal{E}_{\eta,\varepsilon})}(1+k)^{\theta(p-1)}I_k
+ C_{13}(1+k)^\theta+\frac{\alpha}{p}I_k +M_\varepsilon k,
\end{align*}
where, for every $k_\eta>0$,
\begin{equation*}
I_k=\int_\mathcal{W} \frac{\vert  \nabla \mathcal{T}_{k}(v_\varepsilon) \vert ^p}{(1+\vert \mathcal{T}_{k}(v_\varepsilon) \vert)^{\theta (p-1)}}  \;dx, \text{ and } \mathcal{E}_{\eta,\varepsilon}=\lbrace x\in \mathcal{W},\vert  v_\varepsilon(x)\vert > k_\eta \rbrace.
\end{equation*}
Now, we choose $k_{\eta}$ (for fixed $\eta=\bar{\eta}$) such that
\begin{equation*}
\frac{p^\prime}{\alpha}C_{12}\Vert c_0(x)\Vert_{L^{\frac{N}{p-1}}(\mathcal{E}_{\bar{\eta},\varepsilon})}(1+k)^{\theta(p-1)}\leq \frac{1}{2}.
\end{equation*}
Thus, for every $k>0$, we obtain
\begin{align*}
\ds I_k \leq C_{13}(1+k)^{\theta } +C_{14} k,
\end{align*}
which implies that
\begin{align*}
\ds \int_\mathcal{W} \vert \nabla \mathcal{T}_{k}(v_\varepsilon)\vert^p\;dx&\leq (1+k)^{\theta(p-1)}I_k,\\&
\leq C_{13}(1+k)^{\theta p}+C_{14} k(1+k)^{\theta(p-1)}.
\end{align*}
Noting that, for a suitable constant $\mathcal{C}$ depending on $p$ and $\theta$, we have
\begin{equation*}
k(1+k)^{\theta(p-1)} \leq \mathcal{C}\left(1+k^{\theta(p-1)+1}\right) \quad\quad \forall\, k > 0.
\end{equation*}
So that
\begin{align*}
\ds \int_\mathcal{W} \vert \nabla \mathcal{T}_{k}(v_\varepsilon)\vert^p\;dx&\leq C_{13}(1+k)^{\theta p}+C_{15}+C_{16} k^{\theta(p-1)+1}.
\end{align*}
Finally, being $\theta \in (0,1)$, we use Young's inequality with the exponents 
\begin{equation*}
\kappa=\ds\frac{\theta(p-1)+1}{\theta p} \quad\text{and}\quad \kappa^\prime=\ds\frac{\kappa}{\kappa-1},
\end{equation*}
 we obtain
\begin{equation}\label{uepsilon}
\displaystyle \int_\mathcal{W} \vert  \nabla \mathcal{T}_{k}(v_\varepsilon) \vert ^p \;dx 
\leq L+Mk^{\theta (p-1)+1},\quad \forall k>0,
\end{equation}
where 
\begin{equation*}
\ds M= C_{16} +\frac{1}{\kappa}\quad\text{and}\quad  \ds L=C_{13}+C_{15}+\frac{C_{13}^{\kappa^\prime}}{\kappa^\prime}.
\end{equation*}
Bearing in mind that  $\theta<1$ and the estimate \eqref{uepsilon}, we can apply Lemma \ref{lapes} to extrapolate that the estimates \eqref{est2} and \eqref{est1} are hold. So, the proof of the proposition \ref{pr1}  is now complete.
\end{proof}
Based on the arguments in \cite{38,28}, it can be inferred from estimates \eqref{alphauepsilon} and \eqref{uepsilon} that, for a subsequence which is still indexed by $\varepsilon$, the following convergences hold:
\begin{align}
&v_{\varepsilon} \rightarrow v \text { a.e. in } \mathcal{W},\label{ucv} \\
&\widetilde{\varrho}( v_\varepsilon) \rightarrow \widetilde{\varrho}(v) \text { a.e. in } \mathcal{W},\label{altcv} \\
&\mathcal{T}_{k}\left(v_{\varepsilon}\right) \rightharpoonup \mathcal{T}_{k}(v) \text { weakly in }  W_{0}^{1, p}(\mathcal{W}),\label{tkcv} \\
&\mathcal{T}_{k}\left(\widetilde{\varrho}( v_\varepsilon)\right) \rightharpoonup \mathcal{T}_{k}(\widetilde{\varrho}(v) ) \text { weakly in }  W_{0}^{1, p}(\mathcal{W}),\label{tkaltcv} \\
&b_\varepsilon\left(x, v_\varepsilon, \nabla v_{\varepsilon}\right) \chi_{\left\{\left\vert u_{\varepsilon}\right\vert  \leq k\right\}} \rightharpoonup \sigma_{k} \text { weakly in }(L^{p^{\prime}}(\mathcal{W}))^{N},\label{sigma}
\end{align}

\textbf{$\bullet$ Third step: Energy formula}

Now we look for the following energy estimate of the approximating solutions $v_\varepsilon$
\begin{equation}\label{energy formula}
\lim_{n\rightarrow +\infty}\limsup_{\varepsilon \rightarrow 0}\frac{1}{n}\displaystyle\int_{\lbrace \vert \widetilde{\varrho}(v_\varepsilon)\vert \leq n\rbrace}\varrho(v_\varepsilon) b_\varepsilon(x,v_\varepsilon,\nabla v_\varepsilon)\nabla v_\varepsilon\;dx=0.
\end{equation}
Taking $v=\ds\frac{1}{n}\mathcal{T}_{n}(\widetilde{\varrho}( v_\varepsilon))$ as test function in $\eqref{vpap}$ yields that
 \small{
 \begin{align*}
 &\frac{1}{n}\displaystyle \int_\mathcal{W}\left( b_\varepsilon(x,v_\varepsilon,\nabla v_\varepsilon)+\mathcal{B}_\varepsilon(x,v_\varepsilon)\right) \nabla \mathcal{T}_{n}(\widetilde{\varrho}( v_\varepsilon))\;dx\\&=\frac{\gamma}{n} \displaystyle\int_{\mathcal{W}} \frac{\left\vert \mathcal{T}_{\frac{1}{\varepsilon}}(v_\varepsilon)\right\vert ^{s-1}\mathcal{T}_{\frac{1}{\varepsilon}}(v_\varepsilon)}{\vert x\vert ^p+\varepsilon}\mathcal{T}_{n}(\widetilde{\varrho}( v_\varepsilon))\;dx+\frac{1}{n} \int_{\mathcal{W}} f^\varepsilon  \mathcal{T}_{n}(\widetilde{\varrho}( v_\varepsilon))\;dx.
 \end{align*} }
 Utilizing $\eqref{jjj},$ $\eqref{b3},$ $\eqref{b7},$  as well as Hölder's, Young's, and Sobolev's inequalities, we obtain
 \small{ 
 \begin{align*}
 \frac{\alpha}{n}\displaystyle \int_\mathcal{W} \vert  \nabla \mathcal{T}_{n}(\widetilde{\varrho}( v_\varepsilon)) \vert ^p\;dx &\leq \ds\frac{1}{n} \int_{\mathcal{W}} c_0(x)\vert v_\varepsilon \vert^\lambda \vert \nabla \mathcal{T}_{n}(\widetilde{\varrho}( v_\varepsilon)) \vert\;dx+\frac{\gamma}{n} \displaystyle\int_{\mathcal{W}} \frac{\vert v_\varepsilon\vert ^{s}}{\vert x\vert ^p}\mathcal{T}_{n}(\widetilde{\varrho}(v_\varepsilon))\;dx\\&+\frac{1}{n} \int_{\mathcal{W}} f^\varepsilon  \mathcal{T}_{n}(\widetilde{\varrho}( v_\varepsilon))\;dx\\
 & \leq \ds\frac{1}{n}\int_{\mathcal{W}} c_0(x)\frac{\vert v_\varepsilon \vert^\lambda}{\left(1+\vert \widetilde{\varrho}( v_\varepsilon)\vert\right)^{p-1}}\left(1+\vert \widetilde{\varrho}( v_\varepsilon)\vert\right)^{p-1} \vert\nabla \widetilde{\varrho}( v_\varepsilon)\vert\;dx\\
 &+\frac{\gamma}{n} \displaystyle\int_{\mathcal{W}} \frac{\vert v_\varepsilon\vert ^{s}}{\vert x\vert ^p}\mathcal{T}_{n}(\widetilde{\varrho}(v_\varepsilon))\;dx+\frac{1}{n} \int_{\mathcal{W}} f^\varepsilon  \mathcal{T}_{n}(\widetilde{\varrho}( v_\varepsilon))\;dx\\
 &\leq \frac{C_{17}}{n}+\frac{\alpha}{np}\int_\mathcal{W} \vert  \nabla \mathcal{T}_{n}(\widetilde{\varrho}( v_\varepsilon)) \vert ^p\;dx\\&+\frac{1}{n}2^{p-2}\widetilde{C}\mathcal{S}^{p-1}\big\Vert  c_0(x)\big\Vert _{L^{\frac{N}{p-1}}(\bar{Z}_{\eta,\varepsilon}^c)}\int_{\mathcal{W}} \vert \nabla \mathcal{T}_{k}(\widetilde{\varrho}( v_\varepsilon))\vert ^p\; dx\\
 &+\frac{\gamma}{n} \displaystyle\int_{\mathcal{W}} \frac{\vert v_\varepsilon\vert ^{s}}{\vert x\vert ^p}\mathcal{T}_{n}(\widetilde{\varrho}(v_\varepsilon))\;dx+\frac{1}{n} \int_{\mathcal{W}} f^\varepsilon  \mathcal{T}_{n}(\widetilde{\varrho}( v_\varepsilon))\;dx.
 \end{align*}}
with $\bar{Z}_{\eta,\varepsilon}$ is the measurable set defined by \eqref{ZZ}.

By choosing $\eta=\overline{\eta}$ (since $\widetilde{\varrho}( v_\varepsilon)$ is finite a.e. in $\mathcal{W}$), such that
\begin{equation*}
\ds\frac{2^{p-2}p^{\prime}}{\alpha}\widetilde{C}\mathcal{S}^{p-1}\left\Vert  c_0(x)\right\Vert _{\ds L^{\frac{N}{p-1}}(\bar{Z}_{\bar{\eta},\varepsilon}^c)}\leq \frac{1}{2},
\end{equation*}
  we obtain 
 \begin{align*}
 \frac{1}{n}\displaystyle \int_\mathcal{W} \vert  \nabla \mathcal{T}_{n}(\widetilde{\varrho}( v_\varepsilon)) \vert ^p\;dx &\leq \frac{p^\prime C_{17}}{n\alpha}+\frac{\gamma p^\prime}{n \alpha}\displaystyle\int_{\mathcal{W}} \frac{\vert v_\varepsilon\vert ^{s}}{\vert x\vert ^p}\mathcal{T}_{n}(\widetilde{\varrho}(v_\varepsilon))\;dx+\frac{p^\prime}{n \alpha}\int_{\mathcal{W}} f^\varepsilon  \mathcal{T}_{n}(\widetilde{\varrho}( v_\varepsilon))\;dx.
 \end{align*}
 Now, we claim that
 \begin{equation}\label{f1}
 \lim_{n\rightarrow +\infty}\limsup_{\varepsilon \rightarrow 0}\frac{p^\prime}{n \alpha}\int_{\mathcal{W}} f^\varepsilon  \mathcal{T}_{n}(\widetilde{\varrho}( v_\varepsilon))\;dx=0,
 \end{equation}
 and
 \begin{equation}\label{f3}
 \lim_{n\rightarrow +\infty}\limsup_{\varepsilon \rightarrow 0}\frac{\gamma p^\prime}{n \alpha}\displaystyle\int_{\mathcal{W}} \frac{\vert v_\varepsilon\vert ^{s}}{\vert x\vert ^p}\mathcal{T}_{n}(\widetilde{\varrho}(v_\varepsilon))\;dx=0.
 \end{equation}
 Once this claim is proved and by using \eqref{b3} , it follows that \eqref{energy formula} is hold.
 
 For any $n\in\N^*$ and in view of \eqref{ucv} we have
 \begin{equation}
 \ds \mathcal{T}_{n}(v_\varepsilon)\rightharpoonup \mathcal{T}_{n}(v),\quad \text{weak-* in $L^\infty(\mathcal{W}),$}
 \end{equation}
 which implies, combining with \eqref{cvdf}, that
 \begin{equation*}
 \limsup_{\varepsilon \rightarrow 0}\frac{1}{n}\int_{\mathcal{W}} f^\varepsilon  \mathcal{T}_{n}(\widetilde{\varrho}( v_\varepsilon))\;dx=\frac{1}{n}\int_{\mathcal{W}} f \mathcal{T}_{n}(\widetilde{\varrho}(v))\;dx.
 \end{equation*}
Moreover, due to \eqref{ae}, \eqref{altcv} and Fatou Lemma, we  can easily get that
\begin{align*}
meas \left\lbrace \vert \widetilde{\varrho}(v)\vert > \bar{\sigma}_\eta \right\rbrace
&\leq\liminf_{\varepsilon \rightarrow 0}  meas \left\lbrace \vert \widetilde{\varrho}(v_\varepsilon)\vert > \bar{\sigma}_\eta \right\rbrace\leq \ds\frac{1}{\eta^p},
\end{align*}
 thus, it follows  that $\widetilde{\varrho}(v)$ is  finite a.e. in $\mathcal{W}$. In addition, the sequence 
 $\ds\left\lbrace \frac{\mathcal{T}_{n}(\widetilde{\varrho}(v))}{n}\right\rbrace$ converges to 0 a.e. in $\mathcal{W}$ and bounded by 1. 
 Hence, applying Lebesgue's dominated convergence theorem leads to \eqref{f1}.
  
Defining $E$ as a measurable subset of $\mathcal{W}$ containing $0$ and such that the $meas(E)$ is small enough, and according to \eqref{est1} it yields that
\begin{align*}
\displaystyle \int_E \frac{\vert v_\varepsilon\vert ^{s}}{\vert x\vert ^p}dx&\leq \left(\displaystyle\int_\mathcal{W} \vert  v_\varepsilon \vert ^{(1-\theta )(p-1)}\;dx\right)^{\frac{s}{(1-\theta )(p-1)}}\left(\displaystyle\int_E \frac{dx}{\vert  x \vert ^{\frac{p(1-\theta )(p-1)}{(1-\theta )(p-1)-s}}}\right)^{\frac{(1-\theta )(p-1)-s}{(1-\theta )(p-1)}},\\
  &\leq  \left\Vert  \left\vert  v_\varepsilon \right\vert ^{(1-\theta )(p-1)}\right\Vert _{L^{1}(\mathcal{W})}^{\frac{s}{(1-\theta )(p-1)}}\left(\displaystyle\int_E \frac{dx}{\vert  x \vert ^{\frac{p(1-\theta )(p-1)}{(1-\theta )(p-1)-s}}}\right)^{\frac{(1-\theta )(p-1)-s}{(1-\theta )(p-1)}},\\
&\leq C_{18} \left(\displaystyle\int_E \frac{dx}{\vert  x \vert ^{\frac{p(1-\theta )(p-1)}{(1-\theta )(p-1)-s}}}\right)^{\frac{(1-\theta )(p-1)-s}{(1-\theta )(p-1)}}.
\end{align*}
Then, for every $ \delta > 0$ and since $\ds s < \frac{p(1-\theta)(p-1)}{p^*}$, the absolute continuity of the Lebesgue integral allow us to conclude that
\begin{equation*}
\left(\displaystyle\int_E \frac{dx}{\vert  x \vert ^{\frac{p(1-\theta )(p-1)}{(1-\theta )(p-1)-s}}}\right)^{\frac{(1-\theta )(p-1)-s}{(1-\theta )(p-1)}}\leq \delta,
\end{equation*}
therefore, we get 
\begin{equation*}
\displaystyle \int_E \frac{\vert v_\varepsilon\vert ^{s}}{\vert x\vert ^p}dx \leq C_{18}\delta,
\end{equation*}
then, the sequence $\ds\left\lbrace  \frac{\left\vert v_\varepsilon\right\vert ^{s}}{\vert x\vert ^p} \right\rbrace$  is equi-integrable.

This implies, by applying Vitali's theorem,  that
\begin{equation}\label{h1}
\frac{\left\vert v_\varepsilon\right\vert ^{s}}{\vert x\vert ^p}\rightarrow \frac{\vert v\vert ^{s}}{\vert x\vert ^p} \quad \text{strongly in  $L^1(\mathcal{W})$},
\end{equation}
 and thus \eqref{f3} holds, since the sequence $\ds\left\lbrace \frac{\mathcal{T}_{n}(\widetilde{\varrho}(v))}{n}\right\rbrace$ converges to 0 weak-* in $L^\infty(\mathcal{W})$. As a conclusion the energy formula \eqref{energy formula} is proved.

\textbf{$\bullet$ Fourth step: The a.e. convergence  of the sequence $\nabla v_\varepsilon$}

  The main point of this step is  proving the a.e. convergence of $\nabla v_\varepsilon$ in $\mathcal{W}$ and this is will be done by using an arguments similar to these used in \cite{104}. Note that, here we use a slightly different techniques due to the existence of the convection term $-\operatorname{div}(\mathcal{B}(x,v))$ and the term of Hardy potential $\ds\gamma \frac{\vert v\vert^{s-1}v}{\vert x\vert^p}$ in our operator.
\begin{lemma}\label{t2}
Assuming $v_\varepsilon$ is a sequence of solutions to the problems \eqref{problemapr}
with $f^{\varepsilon}$ strongly converging to some $f$ in $L^{1}(\mathcal{W}).$ Suppose that:

(i)  $\mathcal{T}_{k}\left(v_{\varepsilon}\right)$ belongs to $W_{0}^{1, p}(\mathcal{W})$ for every $k>0$,

(ii) $v_\varepsilon$ converges almost everywhere in $\mathcal{W}$ to some measurable function $v$ which is finite almost everywhere and $\mathcal{T}_{k}(v)$ belongs to $W_0^{1,p}(\mathcal{W})$ for every $k > 0$,

(iii) $\vert v_\varepsilon\vert^{(1-\theta)(p-1)}$ is bounded in $L^{\frac{N}{N-p},\infty}(\mathcal{W})$, and $\vert v\vert^{(1-\theta)(p-1)}$ belongs to $L^{\frac{N}{N-p},\infty}(\mathcal{W})$, and

(iv) $\left\vert \nabla v_{\varepsilon}\right\vert ^{{(1-\theta)(p-1)}}$ is bounded in $L^{\frac{N}{N-1-\theta(p-1)},\infty}(\mathcal{W})$, and $\vert \nabla v\vert ^{(1-\theta)(p-1)}$ belongs to $L^{\frac{N}{N-1-\theta(p-1)},\infty}(\mathcal{W})$.

Then, up to a subsequence, $\nabla v_{\varepsilon}$ converges almost everywhere in $\mathcal{W}$ to $\nabla v$, the weak gradient of $v$.
\end{lemma}
\begin{proof}
 Let $\sigma>1$ and $\tau>1$ such that 
 \begin{equation}\label{m1}
 0<\sigma \tau <\frac{N(1-\theta)(p-1)}{p(N-1-\theta(p-1))}.
 \end{equation}
    Let us consider for any $0<j<k$ the sets
 \begin{equation*}
 C_{k}=\{x \in \mathcal{W}:\vert v(x)\vert  \leq k\}, \quad D_{\varepsilon, k,j}=\{x \in \mathcal{W}:\vert v_\varepsilon(x)-\mathcal{T}_{k}(v(x))\vert \leq j\}.
 \end{equation*}
 And we define
\begin{align*}
I(\varepsilon)&=\ds \int_{\mathcal{W}}\left\{\left[b_\varepsilon\left(x, v_{\varepsilon}, \nabla v_{\varepsilon}\right)-b_\varepsilon(x,v_\varepsilon,\nabla v)\right] \nabla\left(v_{\varepsilon}-v\right)\right\}^{\sigma },\\
&=\int_{C_{k}^c}\left\{\left[b_\varepsilon\left(x, v_{\varepsilon}, \nabla v_{\varepsilon}\right)-b_\varepsilon(x,v_\varepsilon,\nabla v)\right] \nabla\left(v_{\varepsilon}-v\right)\right\}^{\sigma }\\&+
\int_{C_{k}}\left\{\left[b_\varepsilon\left(x, v_{\varepsilon}, \nabla v_{\varepsilon}\right)-b_\varepsilon(x,v_\varepsilon,\nabla v)\right] \nabla\left(v_{\varepsilon}-v\right)\right\}^{\sigma },\\&
=I_1(\varepsilon,k)+I_2(\varepsilon,k).
\end{align*}
 By Hölder's inequality and  the growth condition $\eqref{b4}$ we get
\begin{equation*}
I_1(\varepsilon,k)\leq \ds C_8 \left(\int_{C_{k}^c} 1+ \vert \nabla v_{\varepsilon} \vert^{\sigma p\tau} +\vert \nabla v\vert ^{\sigma p\tau}\;dx\right)^{\frac{1}{\tau}} meas\lbrace C_{k}^c \rbrace^{1-\frac{1}{\tau}},
\end{equation*}
 We now choose $\sigma $ and $r$ such that \eqref{m1} is hold, therefore putting together $(iv)$ and  the inclusion \eqref{ceb} (with $r=\frac{N(1-\theta)(p-1)}{N-1-\theta(p-1)}$ and $q=\sigma p \tau$), we deduce that
 \begin{equation*}
I_1(\varepsilon,k)\leq \ds C_8  meas\lbrace C_{k}^c \rbrace^{1-\frac{1}{\tau}}.
\end{equation*} 
By (iii), and by the choice of $\sigma$ , we thus have
\begin{equation}\label{cvv1}
\lim _{k \rightarrow+\infty} \limsup _{\varepsilon \rightarrow 0} I_1(\varepsilon, k)=0 \text {. }
\end{equation}
For $j$ fixed, we have 
\begin{align*}
I_2(\varepsilon,k)&\leq \ds \int_{\mathcal{W}}\left\{\left[b_\varepsilon\left(x, v_{\varepsilon}, \nabla v_{\varepsilon}\right)-b_\varepsilon(x,v_\varepsilon,\nabla \mathcal{T}_{k}( v))\right] \nabla\left(v_{\varepsilon}-\mathcal{T}_{k}(v)\right)\right\}^{\sigma },\\&\leq \ds \int_{D_{\varepsilon,k,j}^c}\left\{\left[b_\varepsilon\left(x, v_{\varepsilon}, \nabla v_{\varepsilon}\right)-b_\varepsilon(x,v_\varepsilon,\nabla \mathcal{T}_{k}( v))\right] \nabla\left(v_{\varepsilon}-\mathcal{T}_{k}(v)\right)\right\}^{\sigma }
\\&+\ds \int_{D_{\varepsilon,k,j}}\left\{\left[b_\varepsilon\left(x, v_{\varepsilon}, \nabla v_{\varepsilon}\right)-b_\varepsilon(x,v_\varepsilon,\nabla \mathcal{T}_{k}( v))\right] \nabla\left(v_{\varepsilon}-\mathcal{T}_{k}(v)\right)\right\}^{\sigma }\\&+
I_3(\varepsilon,k)+I_4(\varepsilon,k),
\end{align*} 
thanks to (iii), we can affirm that
\begin{align*}
\lim _{k \rightarrow+\infty}\limsup_{\varepsilon \rightarrow 0} meas\lbrace D_{\varepsilon,k,j}^c \rbrace\leq \lim _{k \rightarrow+\infty} meas\lbrace\vert v-\mathcal{T}_{k}(v)\vert> j \rbrace=0 ,
\end{align*}
thus, reasoning as for $I_1(\varepsilon,k)$, one has
\begin{equation}\label{cvv2}
\lim _{k \rightarrow+\infty} \limsup _{\varepsilon \rightarrow 0} I_3(\varepsilon, k)=0 \text {. }
\end{equation}

For $I_4(\varepsilon,k)$, one can rewrite it as
\begin{align*}
I_4(\varepsilon,k)=\ds \int_{\mathcal{W}}\left\{\left[b_\varepsilon\left(x, v_{\varepsilon}, \nabla v_{\varepsilon}\right)-b_\varepsilon(x,v_\varepsilon,\nabla \mathcal{T}_{k}( v))\right] \nabla T_j\left(v_{\varepsilon}-\mathcal{T}_{k}(v)\right)\right\}^{\sigma }.
\end{align*}
By the Hölder's inequality (with exponents $\frac{1}{\sigma}$ and $\frac{1}{1-\sigma}$), we have
\begin{align*}
I_4(\varepsilon,k)&\leq \left(\int_\mathcal{W}\left[b_\varepsilon\left(x, v_{\varepsilon}, \nabla v_{\varepsilon}\right)-b_\varepsilon(x,v_\varepsilon,\nabla \mathcal{T}_{k}(v))\right] \nabla T_j\left(v_{\varepsilon}-\mathcal{T}_{k}(v)\right)\;dx\right)^{\sigma } meas(\mathcal{W})^{1-\sigma }.
\end{align*}
 To control the integral on the right-hand side of the previous inequality, we employ $T_j(v_\varepsilon - \mathcal{T}_{k}(v))$ as the test function in \eqref{vpap}. This leads to
 \begin{align*}
&\displaystyle \int_\mathcal{W} b_\varepsilon(x,v_\varepsilon,\nabla v_\varepsilon) \nabla T_j(v_\varepsilon-\mathcal{T}_{k}(v))\;dx+\displaystyle \int_\mathcal{W} \mathcal{B}_\varepsilon(x,v_\varepsilon) \nabla T_j(v_\varepsilon-\mathcal{T}_{k}(v))\;dx\\&= \int_{\mathcal{W}} f^\varepsilon  T_j(v_\varepsilon-\mathcal{T}_{k}(v))\;dx
+\gamma \displaystyle\int_{\mathcal{W}} \frac{\vert \mathcal{T}_{\frac{1}{\varepsilon}}(v)\vert ^{s-1}\mathcal{T}_{\frac{1}{\varepsilon}}(v)}{\vert x\vert ^p+\varepsilon}T_j(v_\varepsilon-\mathcal{T}_{k}(v))\;dx. 
\end{align*}
After some simple calculations, we derive that
\small{
\begin{equation*}
\begin{aligned}
&\ds \int_\mathcal{W} \left(b_\varepsilon(x,v_\varepsilon,\nabla v_\varepsilon)-b_\varepsilon(x,v_\varepsilon,\nabla \mathcal{T}_{k}(v))\right) \nabla T_j(\mathcal{T}_{k}(v_\varepsilon)-\mathcal{T}_{k}(v))\;dx\\&\leq j\gamma \displaystyle\int_{\mathcal{W}} \frac{\vert v_\varepsilon\vert ^{s}}{\vert x\vert ^p}\;dx+j\int_{\mathcal{W}} \vert f^{\varepsilon}\vert\;dx - \int_{\mathcal{W}}\mathcal{B}_\varepsilon(x,v_\varepsilon)\nabla T_j(v_\varepsilon-\mathcal{T}_{k}(v))\;dx\\&-
\int_\mathcal{W} a_\varepsilon(x,v_\varepsilon,\nabla \mathcal{T}_{k}(v))\nabla T_j(\mathcal{T}_{k}(v_\varepsilon)-\mathcal{T}_{k}(v))\;dx.
\end{aligned}
\end{equation*}}
Let us pass to the limit, as $j$ and $\varepsilon$ tend to 0, in all the term of the right hand side of the above inequality.  First of all, thanks to the properties of $f^\varepsilon$  we have
\begin{equation*}
\lim _{j \rightarrow 0}\lim _{\varepsilon \rightarrow 0} j\int_{\mathcal{W}} \vert f^{\varepsilon}\vert\;dx=\lim _{j \rightarrow 0}j\int_{\mathcal{W}} \vert f \vert \;dx=0.
\end{equation*}
From \eqref{h1}, we conclude that
\begin{equation*}
\lim _{j \rightarrow 0}\lim _{\varepsilon \rightarrow 0} j\gamma \displaystyle\int_{\mathcal{W}} \frac{\vert v_\varepsilon\vert ^{s}}{\vert x\vert ^p}\;dx=\lim _{j \rightarrow 0} j\gamma \displaystyle\int_{\mathcal{W}} \frac{\vert v\vert ^{s}}{\vert x\vert ^p}\;dx=0.
\end{equation*}
Noting that
\begin{align*}
 \int_{\mathcal{W}}\mathcal{B}_\varepsilon(x,v_\varepsilon)\nabla T_j(v_\varepsilon-\mathcal{T}_{k}(v))\;dx=\int_{\mathcal{W}}\mathcal{B}_\varepsilon(x,T_{j+k}(v_\varepsilon))\nabla T_j(T_{j+k}(v_\varepsilon)-\mathcal{T}_{k}(v))\;dx,
\end{align*}
therefore, by the growth condition \eqref{b7}, we obtain that
\begin{equation*}
\vert \mathcal{B}_\varepsilon(x,T_{j+k}(v_\varepsilon)) \vert\leq (j+k)^\lambda c_0(x),
\end{equation*}
moreover, by Lebesgue's convergence theorem, we get
\begin{equation*}
\mathcal{B}_\varepsilon(x,T_{j+k}(v_\varepsilon))\rightarrow \mathcal{B}(x,T_{j+k}(v)) \quad \text{strongly in}\quad (L^{p^\prime}(\mathcal{W}))^N,
\end{equation*}
thus, for $\ds \frac{1}{\varepsilon}>j+k$, we have
\begin{equation*}
\lim _{\varepsilon \rightarrow 0} \int_{\mathcal{W}}\mathcal{B}_\varepsilon(x,v_\varepsilon)\nabla T_j(v_\varepsilon-\mathcal{T}_{k}(v))\;dx=\int_{\mathcal{W}}\mathcal{B}(x,T_{j+k}(v))\nabla T_j(T_{j+k}(v)-\mathcal{T}_{k}(v))\;dx,
\end{equation*}
since, for any $j\leq 1$,  we have
\begin{equation*}
\begin{cases}
\nabla T_j(T_{k+j}(v)-\mathcal{T}_{k}(v))\rightarrow 0\quad \text{a.e. as } \;j\rightarrow 0\\
\vert \nabla T_j(T_{k+j}(v)-\mathcal{T}_{k}(v))\vert\leq \vert\nabla T_1(T_{k+1}(v)-\mathcal{T}_{k}(v))\vert\in L^p(\mathcal{W})
\end{cases}
\end{equation*}
then, by Lebesgue's convergence theorem it follows that
\begin{equation*}
\lim _{j \rightarrow 0}\lim _{\varepsilon \rightarrow 0} \int_{\mathcal{W}}\mathcal{B}_\varepsilon(x,v_\varepsilon)\nabla T_j(v_\varepsilon-\mathcal{T}_{k}(v))\;dx=0.
\end{equation*}
Using \eqref{b4} and \eqref{ucv}, we prove that
\begin{equation*}
b_\varepsilon(x,\mathcal{T}_{k}(v_\varepsilon),\nabla \mathcal{T}_{k}(v)) \rightarrow b(x,\mathcal{T}_{k}(v),\nabla \mathcal{T}_{k}(v))\quad \text{strongly in }\; (L^{p^\prime}(\mathcal{W}))^N,
\end{equation*}
putting together the last convergence, $(i)$ and $(ii)$ we arrive that
\begin{equation*}
\lim _{j \rightarrow 0}\lim _{\varepsilon \rightarrow 0} \int_\mathcal{W} b_\varepsilon(x,v_\varepsilon,\nabla \mathcal{T}_{k}(v))\nabla T_j(\mathcal{T}_{k}(v_\varepsilon)-\mathcal{T}_{k}(v))\;dx=0.
\end{equation*}
Hence, we obtain
\begin{equation}\label{cvv3}
\lim _{k \rightarrow+\infty} \limsup _{\varepsilon \rightarrow 0} I_4(\varepsilon, k)=0 \text {. }
\end{equation}
 Combining \eqref{cvv1}, \eqref{cvv2} and \eqref{cvv3}, we deduce
\begin{equation*}
\lim _{\varepsilon \rightarrow 0}I(\varepsilon) \leq 0.
\end{equation*}
Therefore, by \eqref{b5}, we deduce that
\begin{equation*}
\int_{\mathcal{W}}\left\{\left[b_\varepsilon\left(x, v_{\varepsilon}, \nabla v_{\varepsilon}\right)-b_\varepsilon(x,v_\varepsilon,\nabla v)\right] \nabla\left(v_{\varepsilon}-v\right)\right\}^{\sigma } \rightarrow 0
\end{equation*}
that is
\begin{equation}\label{cvpp}
\nabla v_{\varepsilon}(x) \rightarrow \nabla v(x) \quad\text{ a.e. in } \mathcal{W}.
\end{equation}
\end{proof}
\textbf{$\bullet$ Fifth Step : Passing to the limit}

In this step we prove that $v$ is a renormalized solutions of \eqref{problem}. Observe that, due to \eqref{aee}, \eqref{altcv} and Fatou Lemma, we  can easily get that
\begin{align*}
meas \left\lbrace \vert v\vert > k_\eta \right\rbrace
&\leq\liminf_{\varepsilon \rightarrow 0}  meas \left\lbrace \vert v_\varepsilon\vert >k_\eta \right\rbrace\leq \ds\frac{1}{\eta^p},
\end{align*}
 thus, $v$ is  finite a.e. in $\mathcal{W}$ and  \eqref{p0} is proved. Moreover, thanks to \eqref{tkaltcv}, we infer that \eqref{p1} is hold.

Combining proposition \ref{pr1} and lemma \ref{t2}, we get
\begin{equation*}
b_\varepsilon(x,v_\varepsilon,\nabla v_\varepsilon)\rightarrow b(x,v,\nabla v)\quad \text{a.e. in $\mathcal{W},$}
\end{equation*} 
 and under the growth assumption \eqref{b4}, we deduce that
 \begin{equation*}
b_\varepsilon(x,v_\varepsilon,\nabla v_\varepsilon)\rightharpoonup b(x,v,\nabla v)\quad \text{weakly in $(L^{p^\prime}(\mathcal{W}))^N,$}
\end{equation*}
and for any $k>0,$ we have 
\begin{equation*}
b_\varepsilon(x,\mathcal{T}_{k}(v_\varepsilon),\nabla \mathcal{T}_{k}(v_\varepsilon))\rightharpoonup b(x,\mathcal{T}_{k}(v),\nabla \mathcal{T}_{k}(v))\quad \text{weakly in $(L^{p^\prime}(\mathcal{W}))^N.$}
\end{equation*}
Moreover, for any $k_n \in(0,\frac{1}{\varepsilon})$(as in the remark \ref{rmk2}), we can write 
\begin{align*}
 &\frac{1}{n}\displaystyle\int_{\lbrace \vert \widetilde{\varrho}(v_\varepsilon)\vert \leq n\rbrace}\varrho(v_\varepsilon) b_\varepsilon(x,v_\varepsilon,\nabla v_\varepsilon)\nabla v_\varepsilon\;dx\\&= \frac{1}{n}\displaystyle\int_{\mathcal{W}} b_\varepsilon(x,v_\varepsilon,\nabla v_\varepsilon)\nabla \mathcal{T}_{n}(\widetilde{\varrho}\left(v_{\varepsilon}\right))\;dx,\\
&=\frac{1}{n}\displaystyle\int_{\mathcal{W}} b_\varepsilon(x,T_{k_n}(v_\varepsilon),\nabla T_{k_n}(v_\varepsilon))\nabla \mathcal{T}_{n}(\widetilde{\varrho}\left(v_{\varepsilon}\right))\;dx,
\end{align*}
thus, from \eqref{alphauepsilon} and \eqref{energy formula}, it follows that
\begin{align*}
 \lim_{n \rightarrow +\infty}\frac{1}{n}\displaystyle\int_{\mathcal{W}} b(x,T_{k_n}(v),\nabla T_{k_n}(v))\nabla \mathcal{T}_{n}(\widetilde{\varrho}\left(v\right)\;dx=0,
\end{align*}
which gives \eqref{p2}.
%
%

Given $k > 0$ and $\forall n\in \N^*$, denote by  $h_n,$ the truncation function defined as 
\begin{equation*}
h_n(t)=1- \frac{\vert \mathcal{T}_{2n}(t)-\mathcal{T}_{n}(t)\vert}{n}, \quad \forall t\in\R.
\end{equation*}
Now we claim that \eqref{vpap} holds true. Let $h \in W^{1, \infty}(\mathbb{R})$ such that $\operatorname{supp}(h)$ is compact and let $\varphi \in D(\mathcal{W})$. For any $n \in \N^* $ the function $h_{n}\left(\widetilde{\varrho}\left(v_{\varepsilon}\right)\right) h(\widetilde{\varrho}(v)) \varphi$ belongs to $W_0^{1,p}(\mathcal{W})\cap L^{\infty}(\mathcal{W}) $, and then it is an admissible test function in \eqref{vpap}. It yields that
\begin{equation}\label{vpp1}
\begin{aligned}
&\displaystyle\int_{\mathcal{W}} h_n^{\prime}(\widetilde{\varrho}\left(v_{\varepsilon}\right)) h(\widetilde{\varrho}(v)) \varphi \varrho\left(v_{\varepsilon}\right)b_\varepsilon(x,v_\varepsilon,\nabla v_\varepsilon)  \nabla v_\varepsilon \;dx
\\&
+\displaystyle\int_{\mathcal{W}} h_{n}\left(\widetilde{\varrho}\left(v_{\varepsilon}\right)\right)a_\varepsilon(x, v_\varepsilon, \nabla v_\varepsilon)  \nabla [ h(\widetilde{\varrho}(v)) \varphi]\;dx 
\\&
+\displaystyle\int_{\mathcal{W}} \mathcal{B}_\varepsilon(x, v_\varepsilon)\varrho\left(v_{\varepsilon}\right)\nabla v_\varepsilon  h_n^{\prime}(\widetilde{\varrho}\left(v_{\varepsilon}\right)) h(\widetilde{\varrho}(v)) \varphi \;dx
\\&
+\int_{\mathcal{W}} \mathcal{B}_\varepsilon(x, v_\varepsilon)  h_{n}\left(\widetilde{\varrho}\left(v_{\varepsilon}\right)\right)  \nabla [ h(\widetilde{\varrho}(v)) \varphi] \;dx
\\&
=\gamma \displaystyle\int_{\mathcal{W}} \frac{\left\vert \mathcal{T}_{\frac{1}{\varepsilon}}(v_\varepsilon)\right\vert^{s-1} \mathcal{T}_{\frac{1}{\varepsilon}}(v_\varepsilon)}{\vert x\vert^p+\varepsilon}h_{n}\left(\widetilde{\varrho}\left(v_{\varepsilon}\right)\right) h(\widetilde{\varrho}(v)) \varphi \;dx
\\&
+  \int_{\mathcal{W}} f^\varepsilon h_{n}\left(\widetilde{\varrho}\left(v_{\varepsilon}\right)\right) h(\widetilde{\varrho}(v)) \varphi\;dx
\end{aligned}
\end{equation}
Let us pass to limit in \eqref{vpp1} as $\varepsilon$ goes to zero and as $n$ goes to $+\infty$.

Recalling  the definition of function $h_n$, we have
\begin{align*}
&\displaystyle\int_{\mathcal{W}} h_n^{\prime}(\widetilde{\varrho}\left(v_{\varepsilon}\right)) h(\widetilde{\varrho}(v)) \varphi \varrho\left(v_{\varepsilon}\right)b_\varepsilon(x,v_\varepsilon,\nabla v_\varepsilon)  \nabla v_\varepsilon \;dx\\&=\frac{1}{n}\displaystyle\int_{\lbrace n<\vert \widetilde{\varrho}\left(v_{\varepsilon}\right)\vert<2n\rbrace} sign(\widetilde{\varrho}\left(v_{\varepsilon}\right)) h(\widetilde{\varrho}(v)) \varphi \varrho\left(v_{\varepsilon}\right)b_\varepsilon(x,v_\varepsilon,\nabla v_\varepsilon)  \nabla v_\varepsilon \;dx,\\&
\leq\frac{1}{n}\displaystyle\int_{\lbrace n<\vert \widetilde{\varrho}\left(v_{\varepsilon}\right)\vert<2n\rbrace}  h(\widetilde{\varrho}(v)) \varphi \varrho\left(v_{\varepsilon}\right)b_\varepsilon(x,v_\varepsilon,\nabla v_\varepsilon)  \nabla v_\varepsilon \;dx,
\end{align*}
the bounded character of the term $h(\widetilde{\varrho}(v)) \varphi$ and \eqref{energy formula}, allow us to conclude that
\begin{equation}\label{lim1}
 \lim_{n \rightarrow +\infty}\limsup_{\varepsilon \rightarrow 0} \displaystyle\int_{\mathcal{W}} h_n^{\prime}(\widetilde{\varrho}\left(v_{\varepsilon}\right)) h(\widetilde{\varrho}(v)) \varphi \varrho\left(v_{\varepsilon}\right)b_\varepsilon(x,v_\varepsilon,\nabla v_\varepsilon)  \nabla v_\varepsilon \;dx=0.
\end{equation}
Under the growth assumption \eqref{b4}  and according to the lemma \ref{t2}, we have
\begin{equation*}
h_{n}\left(\widetilde{\varrho}\left(v_{\varepsilon}\right)\right)b_\varepsilon(x, v_\varepsilon, \nabla v_\varepsilon)\rightharpoonup h_{n}\left(\widetilde{\varrho}\left(v\right)\right)b(x,v,\nabla v)\quad \text{ weakly in }\;(L^{p^\prime}(\mathcal{W}))^N,
\end{equation*}
In addition of the fact that the function $h(\widetilde{\varrho}(v)) \varphi$ is in $W^{1,p}_0(\mathcal{W})$, we get
\begin{align*}
 \limsup_{\varepsilon \rightarrow 0} &\displaystyle\int_{\mathcal{W}} h_{n}\left(\widetilde{\varrho}\left(v_{\varepsilon}\right)\right)b_\varepsilon(x, v_\varepsilon, \nabla v_\varepsilon)  \nabla [ h(\widetilde{\varrho}(v)) \varphi]\;dx \\&=\displaystyle\int_{\mathcal{W}} h_{n}\left(\widetilde{\varrho}\left(v\right)\right)b(x,v,\nabla v)  \nabla [ h(\widetilde{\varrho}(v)) \varphi]\;dx, 
\end{align*}
moreover, we have that
\begin{equation*}
\begin{aligned}
&h_{n}\left(\widetilde{\varrho}\left(v\right)\right)  \rightarrow 1\quad\text{a.e. in}\;\mathcal{W},\\
&b(x,v,\nabla v)  \nabla [ h(\widetilde{\varrho}(v)) \varphi] \in L^1(\mathcal{W}),\\
&\nabla [ h(\widetilde{\varrho}(v)) \varphi]=\varphi \varrho(v)h^\prime(\widetilde{\varrho}(v)) \nabla v+ h(\widetilde{\varrho}(v))\nabla \varphi,
\end{aligned}
\end{equation*}
so that, Lebesgue's convergence theorem  allows us to derive that
\begin{equation}\label{lim2}
\begin{aligned}
 \lim_{n \rightarrow +\infty}\limsup_{\varepsilon \rightarrow 0} &\displaystyle\int_{\mathcal{W}} h_{n}\left(\widetilde{\varrho}\left(v_{\varepsilon}\right)\right)b_\varepsilon(x, v_\varepsilon, \nabla v_\varepsilon)  \nabla [ h(\widetilde{\varrho}(v)) \varphi]\;dx \\&=\displaystyle\int_{\mathcal{W}} b(x,v,\nabla v)  \varphi \varrho(v)h^\prime(\widetilde{\varrho}(v)) \nabla v\;dx\\&+\int_{\mathcal{W}} b(x,v,\nabla v)  h(\widetilde{\varrho}(v))\nabla \varphi\;dx,\\
 \end{aligned}
\end{equation}
 by means of the definition of $h_n$, the growth assumption \eqref{b7}, Hölder's, Sobolev's inequalities and bounded character of $h(\widetilde{\varrho}(v)) \varphi$, we deduce that
 \begin{equation*}
 \begin{aligned}
 &\displaystyle\int_{\mathcal{W}} \mathcal{B}_\varepsilon(x, v_\varepsilon)\varrho\left(v_{\varepsilon}\right)\nabla v_\varepsilon  h_n^{\prime}(\widetilde{\varrho}\left(v_{\varepsilon}\right)) h(\widetilde{\varrho}(v)) \varphi \;dx\\&
 \leq \frac{1}{n}\Vert h\Vert_{L^\infty(\R)}\Vert \varphi\Vert_{L^\infty(\mathcal{W})}\int_{\lbrace n<\vert \widetilde{\varrho}\left(v_{\varepsilon}\right)\vert<2n\rbrace} \vert\mathcal{B}_\varepsilon(x, v_\varepsilon)\vert \vert\varrho\left(v_{\varepsilon}\right)\vert \vert\nabla v_\varepsilon\vert\;dx,\\
 &\leq \frac{C_{\widetilde{\varrho}}}{n}\Vert h\Vert_{L^\infty(\R)}\Vert \varphi\Vert_{L^\infty(\mathcal{W})} \Vert c_0(x)\Vert_{L^{p^\prime}(\mathcal{W})}\Vert\nabla \mathcal{T}_{2n}(\widetilde{\varrho}( v_\varepsilon))\Vert_{L^{p}(\mathcal{W})}\\&+\frac{C_{\widetilde{\varrho}}\mathcal{S}^{p-1}}{n}\Vert h\Vert_{L^\infty(\R)}\Vert \varphi\Vert_{L^\infty(\mathcal{W})}\Vert c_0(x)\Vert_{L^{\frac{N}{p-1}}(\mathcal{W})}\left\Vert\nabla \mathcal{T}_{2n}(\widetilde{\varrho}( v_\varepsilon))\right\Vert_{L^{p}(\mathcal{W})}^{p},
 \end{aligned}
 \end{equation*}
 where $C_{\widetilde{\varrho}}$ is a positive constant which does not depends on $\varepsilon$.
 
 From the energy formula \eqref{energy formula}, we prove that
 \begin{align}\label{lim3}
 \lim_{n \rightarrow +\infty}\limsup_{\varepsilon \rightarrow 0} \displaystyle\int_{\mathcal{W}} \mathcal{B}_\varepsilon(x, v_\varepsilon)\varrho\left(v_{\varepsilon}\right)\nabla v_\varepsilon  h_n^{\prime}(\widetilde{\varrho}\left(v_{\varepsilon}\right)) h(\widetilde{\varrho}(v)) \varphi \;dx=0.
 \end{align}
 Since $supp(h_n)=[-2n,2n]$, then there exists $k_{2n} \in (0,\frac{1}{\varepsilon})$ such that 
 \small{
 \begin{equation*}
 \int_{\mathcal{W}} \mathcal{B}_\varepsilon(x, v_\varepsilon) h_{n}\left(\widetilde{\varrho}\left(v_{\varepsilon}\right)\right)  \nabla [ h(\widetilde{\varrho}(v)) \varphi]dx=\int_{\mathcal{W}} \mathcal{B}_\varepsilon(x, \mathcal{T}_{k_{2n}}( v_\varepsilon))  h_{n}\left(\widetilde{\varrho}\left(v_{\varepsilon}\right)\right)  \nabla [ h(\widetilde{\varrho}(v)) \varphi]dx.
\end{equation*}}
As a consequence of $\eqref{ae}$, $\eqref{ucv}$ and $\eqref{altcv}$   we obtain that
\begin{equation*}
\mathcal{B}_\varepsilon(x,\mathcal{T}_{k_{2n}}( v_\varepsilon)) h_n(\widetilde{\varrho}\left(v_{\varepsilon}\right))\rightarrow \mathcal{B}(x,v) h_n(\widetilde{\varrho}\left(v\right))\quad \text{a.e. in $\mathcal{W}$},
\end{equation*}  
and by the growth condition $\eqref{b7}$ we deduce that 
\begin{equation*}
\big\vert\mathcal{B}_\varepsilon(x,\mathcal{T}_{k_{2n}}( v_\varepsilon)) h_n(\widetilde{\varrho}\left(v_{\varepsilon}\right)) \big\vert\leq (k_{2n})^\lambda c_0(x) \Vert h_n \Vert_{L^\infty(\R)}\in L^{p^\prime}(\mathcal{W}),
\end{equation*}
therefore, Lebesgue's convergence theorem allow us to conclude that
\begin{equation*}
\mathcal{B}_\varepsilon(x, \mathcal{T}_{k_{2n}}(v_\varepsilon)) h_n(\widetilde{\varrho}\left(v_{\varepsilon}\right))\rightarrow \mathcal{B}(x,v) h_n(\widetilde{\varrho}\left(v\right))\quad \text{strongly in $L^{p^\prime}(\mathcal{W})$},
\end{equation*}
hence for $n$ large enough such that $k_{2n}\geq supp(h)$, we have
\begin{equation}\label{lim4}
\begin{aligned}
 \lim_{n \rightarrow +\infty}\limsup_{\varepsilon \rightarrow 0} &\int_{\mathcal{W}} \mathcal{B}_\varepsilon(x, \mathcal{T}_{k_{2n}}( v_\varepsilon))  h_{n}\left(\widetilde{\varrho}\left(v_{\varepsilon}\right)\right)  \nabla [ h(\widetilde{\varrho}(v)) \varphi] \;dx\\&=\lim_{n \rightarrow +\infty}\int_{\mathcal{W}} \mathcal{B}(x, \mathcal{T}_{k_{2n}}(v))  h_{n}\left(\widetilde{\varrho}\left(v\right)\right)  \nabla [ h(\widetilde{\varrho}(v)) \varphi] \;dx,\\&
 =\int_{\mathcal{W}} \mathcal{B}(x,v)h(\widetilde{\varrho}(v))\nabla\varphi\;dx+\int_{\mathcal{W}}\mathcal{B}(x,v)\varphi \varrho(v)h^\prime(\widetilde{\varrho}(v)) \nabla v\,dx.
 \end{aligned}
\end{equation} 
By combining $\eqref{bn}$, $\eqref{ucv},$  $\eqref{altcv}$ and the fact that $h_n$ is bounded by 1, we can apply Lebesgue convergence theorem to establish that
\small{
\begin{equation}\label{lim5}
\begin{aligned}
 \lim_{n \rightarrow +\infty}\limsup_{\varepsilon \rightarrow 0}& \gamma \displaystyle\int_{\mathcal{W}} \frac{\left\vert \mathcal{T}_{\frac{1}{\varepsilon}}(v_\varepsilon)\right\vert^{s-1} \mathcal{T}_{\frac{1}{\varepsilon}}(v_\varepsilon)}{\vert x\vert^p+\varepsilon}h_{n}\left(\widetilde{\varrho}\left(v_{\varepsilon}\right)\right) h(\widetilde{\varrho}(v)) \varphi \;dx\\&=\gamma \displaystyle\int_{\mathcal{W}} \frac{\vert v\vert^{s-1} v}{\vert x\vert^p} h(\widetilde{\varrho}(v)) \varphi \;dx.
\end{aligned}
\end{equation}}
At last $\eqref{cvdf}$, $\eqref{altcv}$ and the behavior of the sequence $h_n$ together with Lebesgue convergence theorem lead to
\begin{equation}\label{lim6}
\lim_{n \rightarrow +\infty}\limsup_{\varepsilon \rightarrow 0}\int_{\mathcal{W}} f^\varepsilon h_{n}\left(\widetilde{\varrho}\left(v_{\varepsilon}\right)\right) h(\widetilde{\varrho}(v)) \varphi\;dx=\int_{\mathcal{W}} f  h(\widetilde{\varrho}(v)) \varphi\;dx
\end{equation}
Finally, thanks to \eqref{lim1}-\eqref{lim6}, we deduce that for any $h\in W^{1,\infty}(\R)$ with compact support in $\R$, we have
\begin{equation*}
\begin{array}{l}
\displaystyle\int_{\mathcal{W}} b(x,v,\nabla v) \varrho(v)\nabla v h^{\prime}(\widetilde{\varrho}(v)) v\;dx+\displaystyle\int_{\mathcal{W}} b(x,v,\nabla v)  \nabla v h(\widetilde{\varrho}(v))\;dx \\
\quad+\displaystyle\int_{\mathcal{W}} \mathcal{B}(x,v) \varrho(v) \nabla v h^{\prime}(\widetilde{\varrho}(v)) v\;dx+\int_{\mathcal{W}} \mathcal{B}(x,v) \nabla v h(\widetilde{\varrho}(v)) \;dx\\
\quad
=\gamma \displaystyle\int_{\mathcal{W}} \frac{\vert u\vert ^{s-1}u}{\vert x\vert ^p}h(\widetilde{\varrho}(v))v \;dx+  \int_{\mathcal{W}} f h(\widetilde{\varrho}(v)) v \;dx\quad
\end{array}
\end{equation*}
for every  $v \in W_0^{1,p}(\mathcal{W}) \cap L^{\infty}(\mathcal{W})$.

At least the limit $v$ satisfies \eqref{p0}, \eqref{p1}, \eqref{p2} and \eqref{vp}, which asserts that $v$ is a renormalized solution of the problem \ref{problem}, then  the proof of Theorem \ref{t1} is now complete.
\end{proof}
\section*{Acknowledgement(s)}
This work has been done while the first author was  PhD  visiting student at the Department of Mathematics and Computer Science, University of Catania, supported by a  MAECI scholarship of the Italian Government.
 \section*{Declarations}
\textbf{Conflict of interest} The author declare that he has no conflict of interest.


\end{document}